\documentclass[12pt]{amsart}
\usepackage{amsmath}
\usepackage{amsthm}
\usepackage{mathrsfs}
\usepackage{amsfonts} 
\usepackage{graphicx}
\usepackage{hyperref}
\usepackage{epsfig}

\usepackage{amsmath,amssymb,amsfonts,amsthm,amstext,verbatim}
\usepackage{enumerate,color,graphicx,
}
\usepackage{psfrag}  
\usepackage{url}
\usepackage{mdwlist}   
\usepackage{mathtools} 
\usepackage{wasysym}   
\usepackage{mathrsfs}
\usepackage{hyperref}
\usepackage{epsfig}

\numberwithin{equation}{section}

\allowdisplaybreaks

\newtheorem{theorem}{Theorem}[section]

\newtheorem{proposition}[theorem]{Proposition}
\newtheorem{lemma}[theorem]{Lemma}
\newtheorem{remark}[theorem]{Remark}
\newtheorem{definition}[theorem]{Definition}

\newcommand{\ep}{\varepsilon}

\newcommand{\R}{\mathbb{R}}
\newcommand{\N}{\mathbb{N}}
\newcommand{\Z}{\mathbb{Z}}

\newcommand{\dist}{\vert\vert}

\newcommand{\ren}{\mathsf{zR}}

\newcommand{\eps}{\ep}

\newcommand{\sab}{\mathsf{P_{SAB}}}

\newcommand{\saw}{\mathsf{P_{SAW}}}

\newcommand{\sawset}{{\rm SAW}}
\newcommand{\sabset}{{\rm SAB}}

\newcommand{\sfT}{\mathsf{T}}

\newcommand{\unf}{\mathsf{Unf}}

\newcommand{\latt}{\Z^d}
\newcommand{\lattv}{\Z^d}
\newcommand{\latte}{E(\Z^d)}

\newcommand{\sahsw}{\mathsf{P_{SAHSW}}}

\newcommand{\shell}{\mathcal{S}}

\newcommand{\g}{\gamma}

\newcommand{\hang}{\mathsf{hang}}

\newcommand\sS{{\mathcal S}}

\newcommand\si{\sigma}       
     
\newcommand\de{\delta}       
\newcommand\resp{respectively}

\newcommand{\RR}{\mathbb{R}}     
\newcommand{\NN}{\mathbb{N}}     
\newcommand{\ZZ}{\mathbb{Z}}     

\newcommand\la{\lambda}

\theoremstyle{remark}

\newcounter{mycount}

\def\mik{1}
\newcommand\cpsfrag[2]{\ifnum\mik=1\psfrag{#1}{#2}\fi}

\newcommand{\SAW}{\sawset}
\newcommand{\SAP}{\mathrm{SAP}}
\newcommand{\SAB}{\mathrm{SAB}}
\newcommand{\sahswset}{\mathrm{SAHSW}}

\newcommand{\Rfl}{\mathcal R}  

\newcommand{\re}{\mathsf{ren}}

\newcommand{\lex}{hanging}  
\newcommand{\lexm}{\hang} 

\newcommand{\closes}{\text{ closes}}

\newcommand{\lo}{\ell_0}         
\newcommand{\PSAW}{\saw} 
\newcommand{\ESAW}{\mathsf{E_{SAW}}}
\newcommand\Ga{\Gamma}     
\newcommand\ga{\gamma}     
\newcommand\Ptmp{P}  
\newcommand\Wtmp{\rm W}  
\newcommand\Atmp{\rm A}  

\newcommand{\hidden}[1]{}
\renewcommand{\shell}{\varsigma}    

\newcommand\mcond{\, \middle| \,}
\newcommand\cond{\, | \,}

\title[The endpoint of self-avoiding walk]{On the probability that self-avoiding walk ends at a given point}
\date{}

\author[H.~Duminil-Copin, A.~Glazman, A.~Hammond and I.~Manolescu]{Hugo Duminil-Copin, Alexander Glazman, Alan Hammond \\ and Ioan Manolescu}
\address{ Department of Statistics, 
  University of Oxford,
  1 South Parks Road,
  Oxford, OX1 3TG, U.K.}
\address{D\'epartement de Math\'ematiques, Universit\'e de Gen\`eve, 2--4 rue du Li\`evre, Gen\`eve, Switzerland.}
\address{St.\,Petersburg Department of Steklov Mathematical Institute (PDMI RAS).
Fontanka 27, 191023 St. Petersburg, Russia} 

\email{duminil.copin@unige.ch, alexander.glazman@unige.ch, \newline hammond@stats.ox.ac.uk,  ioan.manolescu@unige.ch}

\begin{document}
\begin{abstract}
  We prove two results on the delocalization of the endpoint of a uniform self-avoiding walk on $\ZZ^d$ for $d \geq 2$. 
  We show that the probability that a walk of length $n$ ends at a point $x$ 
  tends to $0$ as $n$ tends to infinity, uniformly in $x$. 
  Also, when $x$ is fixed, with $\dist x \dist = 1$, 
  this probability decreases faster than $n^{-1/4 + \eps}$ for any $\eps >0$.
  This provides a bound on the probability that a self-avoiding walk is a polygon.
\end{abstract}
\maketitle

\section{Introduction}

Flory and Orr \cite{Flory,Orr47}  introduced self-avoiding walk as a model of a long chain of molecules. 
Despite the simplicity of its definition, the model has proved resilient to rigorous analysis. 
While in dimensions $d \geq 5$ 
lace expansion techniques provide a detailed understanding of the model,
and the case $d = 4$ is the subject of extensive ongoing research, 
very little is known for dimensions two and three. 

The present paper uses combinatorial techniques to prove two intuitive results for dimensions $d \geq 2$. 
We feel that the interest of the paper lies not only in its results, but also in techniques employed in the proofs. 
To this end, certain tools are emphasised as they may be helpful in future works as well. 

We mention two results from the early 1960s that stand at the base of our arguments: Kesten's \emph{pattern theorem}, which concerns the local geometry of a typical self-avoiding walk, 
and Hammersley and Welsh's {\em unfolding argument}, which gives a bound on the correction to the exponential growth rate in the number of such walks.

\subsection{The model} 

Let $d \geq 2$. For $u \in \R^d$, let $\dist u \dist$ denote the Euclidean norm of $u$. 
Let $\latte$ denote the set of nearest-neighbour bonds of the integer lattice $\latt$.
A {\em walk} of length $n \in \N$ is a map $\gamma:\{0,\ldots,n \} \to \lattv$ 
such that $(\gamma_i,\gamma_{i+1}) \in \latte$ for each $i \in \{0,\ldots,n-1\}$.
An injective walk is called {\em self-avoiding}. 
Let $\sawset_n$ denote the set of self-avoiding walks of length~$n$ that start at $0$. 
We denote by $\saw_n$ the uniform law on $\sawset_n$, and by $\ESAW_n$ the associated expectation. 
The walk under the law $\PSAW_n$ will be denoted by $\Ga$. 

\subsection{The endpoint displacement of self-avoiding walk}

The law of the endpoint displacement under $\PSAW_n$ 
is a natural object of study in an inquiry into the global geometry of self-avoiding walk.
The displacement is quantified by the Flory exponent $\nu$, specified by the conjectural relation 
$\ESAW_n [\dist \Gamma_n \dist^2] = n^{2\nu+o(1)}$. 

In dimension $d \geq 5$, it is rigorously known that $\nu = 1/2$ (see Hara and Slade \cite{HS91,HS92}). 
When $d=4$, $\nu = 1/2$ is also anticipated, though this case is more subtle from a rigorous standpoint. Recently,
some impressive results have been achieved using a supersymmetric renormalization group approach for continuous-time weakly self-avoiding walk: 
see \cite{BDGS11,BIS09,BS10} and references therein. 

When $d=2$, $\nu = 3/4$ was predicted nonrigorously in \cite{Nie82,Nie84} using the  Coulomb gas formalism, 
and then in \cite{Dup89,Dup90} using Conformal Field Theory. 
It is also known subject to the assumption of existence of the scaling limit and its conformal invariance \cite{LSW}.   
Unconditional rigorous statements concerning the global geometry of the model are almost absent in the low dimensional cases at present. 
In \cite{DCH12}, sub-ballistic behaviour of self-avoiding walk in all dimensions $d \geq 2$ was proved, 
in a step towards the assertion that $\nu < 1$. 

\subsection{Results}

This paper is concerned in part with ruling out another extreme behaviour for endpoint displacement, 
namely that $\Gamma_n$ is close to the origin. 
In \cite{Madras12}, the mean-square displacement of the walk is proved to exceed $n^{4/(3d)}$. 
As we will shortly explain, a variation of that argument shows that
$\PSAW_n(\dist \Gamma_n \dist = 1) \leq \tfrac{2}{3}$ for all $n$ high enough. 
Recently, Itai Benjamini posed the question of strengthening this conclusion to 
$\PSAW_n(\dist \Gamma_n \dist = 1) \to 0$ as $n \to \infty$.
In this article, we confirm this assertion 
and also investigate the related question of the uniform delocalization of $\Ga_n$. 

The first theorem deals with a quantitative bound on $\PSAW_n(\dist \Gamma_n \dist = 1)$.
\begin{theorem}\label{t.nonclosing}
  Let $d \geq 2$. For any $\eps > 0$ and $n$ large enough,
  $$\saw_n \big(  \dist \Gamma_n \dist = 1 \big) \leq n^{-1/4 + \eps}.$$
\end{theorem}

Observe that $\saw_n (  \dist \Gamma_n \dist = 1)$ equals $n p_n/c_n$, where $c_n$ and $p_n$ 
denote the number of self-avoiding walks and self-avoiding polygons 
(such polygons will be formally introduced in Section~\ref{sec:quantitative}, and are considered up to translation) of length $n$. It is thus natural to think that Theorem~\ref{t.nonclosing} may be proved by providing bounds on $c_n$ and $p_n$ separately.
On $\Z^2$, the best such rigorous bounds are  $c_n\geq \mu_c^n$ (this follows from submultiplicativity) and $p_n \le Cn^{-1/2} \mu_c^n$ \cite{Madras95}, where $\mu_c := \lim_n \vert c_n \vert^{1/n}$ is the connective constant. These two bounds do not imply that $\saw_{n} ( \dist \Gamma \dist = 1 )$ tends to 0. A similar issue arises in dimensions 3 and 4. We will therefore try to bound the ratio $np_n/c_n$ directly.

For dimensions $d \geq 5$, it has been proved in \cite[Thm. 1.1]{HS92} that \mbox{$c_n \sim A \mu_c^n$} for some constant $A > 0$.
It is expected (and proved for high enough dimension \cite[Thm. 6.1.3]{MS93}) that $p_n \le C n^{-d/2 -1} \mu_c^n$ for some constant $C$.
This proves that
\begin{align*}
	\saw_n (\dist \Gamma_n \dist = 1) \leq C n^{-d/2},
\end{align*} 
an inequality which is consistent with the Gaussian behaviour of $\Gamma_n$ above dimension $4$. 
See \cite{CS09} for an extended discussion of self-avoiding polygons in high dimensions.  

For dimension $d=2$, 
$\saw_{n} ( \dist \Gamma_n \dist = 1 )$ is conjectured to behave like $n^{-59/32 + o(1)}$ as $n \to \infty$ through odd values. This follows from the well-known predictions $p_n = n^{-5/2 + o(1)} \mu_c^n$ for odd $n$
and $c_n = n^{11/32 + o(1)} \mu_c^n$.
The first of these two estimates may be derived from a conjectural hyperscaling relation \cite[1.4.14]{MS93} 
linking its value to that of the Flory exponent $\nu$; 
the second was first predicted by Nienhuis in \cite{Nie82} and can also be deduced from SLE$_{8/3}$, see \cite[Prediction~5]{LSW}.
\medbreak
We expect that Theorem \ref{t.nonclosing} may be extended as follows: for any {\em fixed} point $x$, 
\begin{align}\label{eq:SAW_to_x}
	\saw_n (\Gamma_n = x) \leq n^{-1/4 + \eps}
\end{align} 
for any $n$ large enough.
Indeed, as conjectured in \cite[Conj. 1.4.1]{MS93}, 
local surgery arguments are expected to show that, for $x \in \Z^d$ with an odd sum of coordinates, 
$\saw_n (\Gamma_n = x)$ is comparable to $\saw_n (\dist \Gamma_n\dist = 1)$. (The mention of parity is necessary because 
a walk of even length may not end one step from the origin.)
For $\Z^2$ this may potentially be shown by a tedious ad-hoc construction. For higher dimension,
significant difficulties occur and it is unclear how to prove this result. 
\medbreak
We will not extend Theorem~\ref{t.nonclosing} to arbitrary $x$, but we prove in the theorem below that $\Ga_n$ is not concentrated by providing uniform (yet non quantitative) bounds on the probability to end at a given point.

\begin{theorem}\label{theorem:SAW_to_x}
  Let $d \geq 2$. As $n \to \infty$,
  $\sup_{x \in \Z^d} \saw_n (\Gamma_n = x) \to 0$. 
\end{theorem}

%
%
%
%

The rest of the paper is structured as follows. 
In Section \ref{sec:pre}, we lay out some of the tools used in the proofs, 
namely the multi-valued map principle, unfolding arguments and the Hammersley-Welsh bound. 
The proofs of Theorems \ref{t.nonclosing} and \ref{theorem:SAW_to_x} 
may be found in Sections \ref{sec:quantitative} and \ref{sec:uniform}, respectively.
In spite of the similarity of the results, the two proofs employ very different techniques, 
and indeed Sections \ref{sec:uniform} and \ref{sec:quantitative} may be read independently of each other. 

In Section \ref{sec:pattern} we define and discuss notions revolving around Kesten's pattern theorem.
Although the material in this section is not original {\em per se}, our presentation is novel and, we hope, may be fruitful for further research. 
An example of a consequence is the following delocalization of the midpoint of self-avoiding walk.
\begin{proposition}\label{prop:midpoint}
  There exists a constant $C > 0$ such that, for $n \in \NN$,
  \begin{align}
    \sup_{x \in \ZZ^d} \PSAW_n\left(\Ga_{\lfloor \frac{n}{2} \rfloor} = x \right) \leq C n^{- \frac12}.
    \nonumber
  \end{align}
\end{proposition}

\noindent \textbf{A comment on notation.}
The only paths which we consider in this paper are self-avoiding walks. 
Thus, we may, and usually will, omit the term ``self-avoiding" in referring to them. 
In the course of the paper, several special types of walk will be considered -- such as self-avoiding bridges 
and self-avoiding half-space walks -- and this convention applies to these objects as well.
 \medbreak
\noindent \textbf{Acknowledgments.}
We thank Itai Benjamini for suggesting the problem to us, and a referee for a very thorough reading of the manuscript. The first, second and fourth authors were supported by the ERC grant AG CONFRA as well as the FNS. The third author was supported by EPSRC grant
EP/I004378/1 and thanks the Mathematics Department of the University of
Geneva for its hospitality during the year 2012--13. 
 
\section{Preliminaries}\label{sec:pre}

\subsection{General notation} 

We denote by $\langle\cdot|\cdot\rangle$ the scalar product on $\RR^d$,
and recall the notation $\dist \cdot \dist$ for the Euclidean norm on this space. 
Let $e_1, \dots, e_d$ be the natural basis of $\ZZ^d$.
We will consider $e_1$ to be the \emph{vertical} direction. 
The cardinality of a set $E$ will be denoted by $|E|$. 
The length of a walk $\gamma$, being the cardinality of its edge-set, will be denoted by $|\gamma|$. 
For $0 \leq i \leq j \leq n$, we write $\ga[i, j] = (\ga_i, \dots, \ga_j).$

For $m, n\in \mathbb{N}$, let $\gamma$ and~$\tilde{\gamma}$ be two walks of lengths~$m$ and~$n$, \resp, neither of which needs to start at~$0$. 
The concatenation~$\gamma\circ\tilde{\gamma}$ of~$\gamma$ and~$\tilde{\gamma}$ is given by
$$
\left(\gamma\circ\tilde{\gamma}\right)_k=
\left\{ 
  \begin{array}{l l}
    \gamma_k  \quad & k\leq m,\\
    \gamma_m+(\tilde{\gamma}_{k-m}-\tilde{\gamma}_0) & m+1\leq k \leq m+n.
  \end{array} \right.
$$

\subsection{The multi-valued map principle}

Multi-valued maps and the multi-valued map principle stated next will play a central role in our analysis.

A multivalued map from a set $A$ to a set $B$ is a map  $\Phi:A\rightarrow \mathfrak P(B)$. 
For $b \in B$, let $\Phi^{-1}(b) = \big\{ a \in A: b\in \Phi(a) \big\}$
and 
\begin{equation}
  \Lambda_\Phi(b)=\sum_{a\in\Phi^{-1}(b)}\frac{1}{|\Phi(a)|}.
\end{equation}
The quantity $\Lambda_\Phi(b)$ may be viewed as a (local) contracting factor of the map, 
as illustrated by the following statement. 

\begin{lemma}[Multi-valued map principle]\label{prop:multivalmap}
  Let $A$ and $B$ be two finite sets and $\Phi:A\rightarrow \mathfrak P(B)$  be a multi-valued map. 
  Then
  $$
  |A|=\sum_{b\in B}\Lambda_\Phi(b)\le |B|\max_{b\in B}\Lambda_\Phi(b) \, .
  $$
\end{lemma}
The proof is immediate. 
We will often apply the lemma in the special situation where, for any $b \in B$, $|\Phi(a)|$ is independent of $a \in \Phi^{-1}(b)$.
Then the contracting factor may be written $\Lambda_{\Phi}(b)=\frac{|\Phi^{-1}(b)|}{|\Phi(a)|}$ for any $a \in \Phi^{-1}(b)$.

\subsection{Unfolding self-avoiding walks}\label{sec:unfold}

An unfolding operation, similar to the one used in \cite{hammersleywelsh} and more recently in \cite{Madras12}, 
will be applied on several occasions. 
We first describe the operation in its simplest form, and then the specific version that we will use. 

For $z \in \ZZ^d$, 
let $\Rfl_z$ be the orthogonal reflection with respect to the plane 
$\{x\in \ZZ^d: \langle x | e_1 \rangle  =  \langle z | e_1 \rangle \}$, i.e. the map such that for any $x\in \ZZ^d$,
$$ \Rfl_z (x) = x +  2\langle z  -   x | e_1 \rangle e_1.$$

For $\gamma \in \SAW_n$,  let $k$ be any index such that
$$ 
\langle\gamma_k | e_1\rangle = \max \{\langle\gamma_j | e_1\rangle : 0 \le j \le n\} \, .
$$
The simplest unfoldings of $\gamma$ are
those walks obtained by concatenating $\ga[0,k]$ and $\Rfl_{\gamma_k}(\ga[k,n])$ for such an index $k$.
The condition on $k$ ensures that any such walk is indeed self-avoiding.
 
Of the numerous choices of the unfolding point index $k$, 
the following seems to be the most suitable for our purpose. 
\begin{definition}\label{d.hang}
  For $\ga \in \SAW_n$,  the \emph{hanging time}
  $\lexm = \lexm(\ga)$ is the index $k \in \{0, \dots, n\}$ 
  for which $\ga_k$ is maximal for the lexicographical order of $\ZZ^d$. 
  We call $\ga_\lexm$ the \emph{\lex\ point} and write $\gamma^1=\ga[0,\lexm]$ and $\gamma^2=\ga[\lexm,n]$.  
\end{definition}

Here are two essential properties of the \lex\ point. 
First, $\ga_\lexm$ depends only on the set of points visited by $\ga$, 
not on the order in which they are visited. 
Second, the lexicographical order of $\ZZ^d$ is invariant under translations;
thus the \lex\ time of $\ga$ is the same as that of any translate of $\ga$. 

In a variation (motivated by technical considerations) of the unfolding procedure, we will sometimes add a short walk between $\ga[0,k]$ and $\Rfl_{\gamma_k}\ga[k,n]$.
The specific unfolding that we will use is defined next. 
\begin{definition}
  For a walk $\ga \in \SAW_n$, define $\unf(\gamma) \in \SAW_{n+1}$ to be the concatenation of $\ga^1$, 
  the walk across the edge $e_1$, and the translation of $\Rfl_{\gamma_\lexm}\ga^2$ by $e_1$.
\end{definition}

In \cite{Madras12}, Madras used an unfolding argument to obtain a lower bound on the mean-square displacement of a uniform self-avoiding walk. 
A simple adaptation of his technique proves that $\PSAW_n(\dist \Gamma_n \dist = 1) \leq \tfrac{2}{3}$ for all $n \geq 3^d + 1$.
We now sketch this argument since, for example, the proof of Theorem~\ref{t.nonclosing} may be viewed as an elaboration. 

To any walk $\ga \in \SAW_n$ with $\dist \gamma_n \dist = 1$, 
choose an axial direction in which $\ga$ has maximal coordinate at least two -- we will assume this to be the $e_1$-direction -- 
and associate to $\ga$ its simple unfolding $b \in \SAW_{n}$, 
given by concatenating the $e_1$-reflection of $\ga^2$ to $\ga^1$.
Any unfolded walk $b$ corresponds to at most two walks $\ga$ with $\dist \gamma_n \dist = 1$.
This is because the level at which the unfolding was done is one among 
$\frac{ \langle b_n | e_1 \rangle - 1}{2} $, 
$\frac{ \langle b_n | e_1 \rangle}{2} $ and 
$\frac{ \langle b_n | e_1 \rangle + 1}{2}$,
of which at most two are integers.

Moreover, since the maximal  $e_1$-coordinate of $\ga$ is at least two, 
the unfolding of $\ga$ is such that $\dist b_n \dist > 1$. 
The choice of axial direction for reflection may be made in such a way that it may be determined from the unfolded walk. 
Thus,  $\PSAW_n \left(\dist \Gamma_n \dist = 1 \right) \leq \frac23$.
 
\subsection{The Hammersley-Welsh bound} 
 
For $n \in \N$, let $c_n = |\SAW_n|$ denote the number of length-$n$ walks that start at the origin.
Write $\mu_c \in (0,\infty)$ for the connective constant, given by $\mu_c = \lim_n c_n^{1/n}$.
Its existence follows from the submultiplicativity inequality $c_{n+m} \leq c_n c_m$ and Fekete's lemma.
Furthermore, these also imply that~$c_n^{1/n}$ converges to~$\mu_c$ from above.
This provides a lower bound on $c_n$. 

The value of $\mu_c$ is not rigorously known for any lattice $\ZZ^d$ with $d \geq 2$.
The only non-trivial lattice for which the connective constant (for unweighted self-avoiding walks) has been rigorously derived is the hexagonal lattice, see \cite{DS10}, using parafermionic observables (also see \cite{BBDDG13,Gla13} for other applications of parafermionic observables, including the computation of the connective constant for a model of weighted walks). 

An upper bound for the growth rate of $c_n$ is provided by the Hammersley-Welsh argument of \cite{hammersleywelsh}
(which is proved by an iterative unfolding procedure).
It states the existence of a constant $c_{HW} > 0$ such that, for all $n \in \NN$, 
\begin{equation}\label{eq.hw}
  \mu_c^n \leq c_n \leq e^{c_{HW} \sqrt{n}} \mu_c^n.
\end{equation}
 
\section{The shell of a walk: definition and applications}\label{sec:pattern}

The shell of a walk, defined next, is a notion that appeared implicitly in Kesten's proof of the pattern theorem in~\cite{kestenone}. 
In the next subsections, we present two consequences of the pattern theorem which may be of some general use.
We illustrate this by the proof of Proposition \ref{prop:midpoint}.

\begin{definition}[Type I/II patterns] 
  A pair of type I and II patterns is a pair of self-avoiding walks $\chi^I$, $\chi^{II}$,  
  both contained in the cube $[0,3]^d$, with the properties that
  \begin{itemize*}
  \item $\chi^I$ and $\chi^{II}$ both visit all vertices of the boundary of $[0,3]^d$,
  \item $\chi^I$ and $\chi^{II}$ both start at   $(3, 1, \dots 1)$ and end at $(3,2, 1 \dots, 1)$,
  \item $|\chi^{II}| = |\chi^{I}| + 2$.
  \end{itemize*}
\end{definition}
Figure \ref{fig:patterns} contains examples of such patterns for $d =2$. 
The existence of such pairs of walks for any dimension $d \geq 2$ may be easily checked, and no details are given here. 
Fix a pair of type I and II patterns for the rest of the paper. 

A pattern $\chi$ is said to occur at step $k$ of a walk $\ga$ 
if $\ga[k, k + |\chi|]$ is a translate of $\chi$.
In such case call the translate of $[0,3]^d$ containing $\ga[k, k + |\chi|]$ a {\em slot} of $\ga$. 
A walk may have several occurrences of both type I and II patterns.  
Note that occurrences of such patterns are necessarily disjoint. 

\begin{figure}
  \begin{center}
    \includegraphics[width=0.4\textwidth]{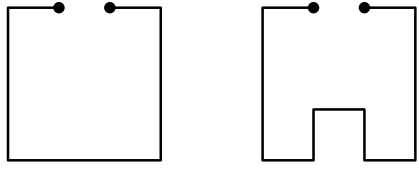}
  \end{center}
  \caption{An example of type I and II patterns for $d =2$.}
  \label{fig:patterns}
\end{figure}

\begin{definition}[Shell of a walk]
  Two self-avoiding walks are equivalent if one can be obtained from the other 
  by changing some type I patterns into type II patterns and some type II patterns into type I patterns.  
  Classes for this equivalence relation are called {\em shells}. 
\end{definition}

All walks of a given shell have the same slots,
thus a shell may be viewed as a walk with certain slots where type I and II patterns may be inserted. 
Note that a shell may contain walks of different lengths, depending on the number of type I and II patterns inserted in the slots.
Shells are convenient to work with as many interesting events may be written in terms of the shell of the walk 
(see below for examples of such events). 

The shell of the walk $\ga$ is denoted $\varsigma(\ga)$; 
$T_I(\ga)$ and $T_{II}(\ga)$ denote the number of occurrences for patterns of type I and II in $\ga$. 
Observe that $T_I(\ga)$ and $T_{II}(\ga)$ are determined by $\varsigma(\ga)$ and the length of $\ga$.
When the random variable $\Ga$ is involved, we will sometimes drop the explicit dependence of $T$ on $\Ga$. 

By Kesten's pattern theorem \cite[Thm. 1]{kestenone}, there exist constants $c > 0$ and $\de >0$ such that,
for any $n \geq d 3^d$,
\begin{equation}\label{eq.patterns}
  \PSAW_n \left( T_I(\Ga) \leq \de n \right) \le e^{-c n} 
  \quad \text{ and } \quad
  \PSAW_n \left( T_{II}(\Ga) \leq \de n \right) \le e^{-c n}.
\end{equation}
  
\subsection{Shell probabilities are stable under perturbation of walk length}

The following lemma states that, when considering typical events expressed only in terms of shells, 
their $\saw_n$ and $\saw_{n + 2k}$ probabilities are comparable for $k$ small enough.
The lemma will be instrumental when applying multi-valued map arguments. 

Let $A$ be a collection of shells and let 
$$
A_n=\{\gamma\in\SAW_n:\varsigma(\gamma)\in A\}.
$$ 
The set $A$ may be chosen as, for example, the set of half-space walks, bridges, irreducible bridges 
or walks with certain constraints on the number of renewal times 
(the definitions of these notions are given below). 
This flexibility in the choice of $A$ may render the lemma useful in a broad context.
We will use it in the proof of Theorem~\ref{theorem:SAW_to_x}, more precisely in Proposition~\ref{prop:M}. 

\begin{lemma}\label{lem:stability}
  For any $c>0$, there exists $C > 0$ such that the following occurs. 
  Let $A$ be a set of shells and~$n$ be an integer such that $|A_n|\ge e^{-c\sqrt n}\mu_c^n$. 
  Then, for any $0\le k\le n^{1/5}$ such that~$A_{n-2k}\ne\varnothing$,
  \begin{equation}\label{eq.stability}
    |A_{n-2k}|\ge C |A_n|\mu_c^{-2k}.
  \end{equation}
\end{lemma}

The argument by which we will derive the lemma is a direct adaptation of one in \cite{kestenone}, where the claim is proved for $A_n=\SAW_n$. 
In this case, the result can be improved to $|k|\le n^{1/3}$ 
due to the submultiplicativity property of the number of self-avoiding walks.  
For the uses we have in mind, $n^{1/5}$ is sufficient. 

\begin{proof}
  Fix $c >0$. It suffices to prove the statement for $n$ large enough
  (the specific requirements on $n$ will be indicated in the proof). Fix a value $n$. 

  Let $\de, c_\de > 0$ be constants such that \eqref{eq.patterns} holds. 
  For $\ell \in \Z$, define
  \begin{align*}
    \tilde A_{n+2\ell} &=\big\{ \gamma\in A_{n+2\ell}:T_{II}(\gamma)>\delta n+\ell \big\}.
  \end{align*} 
  Note that the assumption $|A_n|\ge e^{-c\sqrt n}\mu_c^n$ and the choice of $\de$ yield
  $$ |A_n\setminus\tilde A_n|\le e^{-c_{\de}n} |SAW_n| \le e^{-c_\de n}e^{(c_{HW} + c)\sqrt{n}} |A_n|.$$
  As a consequence, for $n$ larger than some value depending on $c_\de, c_{HW}$ and $c$, 
  \begin{equation}\label{eq:ab}
    |\tilde A_n|\ge \tfrac12 |A_n| \ge  \tfrac12 e^{-c\sqrt n}\mu_c^n.
  \end{equation}
  This will be useful later, and henceforth we assume that $n$ is sufficiently large for \eqref{eq:ab} to hold. Let us also assume that $\de n \geq n^{3/4}$.
  \medbreak 

  We start by proving that, for any $\ell \in \N$ with $| \ell | \leq n^{3/4}$, when setting $m = n + 2\ell$, 
  \begin{equation}\label{eq:crucial}
    \frac{|\tilde A_{m+2}|}{|\tilde A_m|}-\frac{|\tilde A_{m+4}|}{|\tilde A_{m+2}|}\le \frac{2}{\delta^3n}\cdot\frac{|\tilde{A}_m|}{|\tilde{A}_{m+2}|}.
  \end{equation}  
  Consider the multi-valued map from $\tilde A_{m+2}$ into $\tilde A_m$ 
  that consists of replacing a type II pattern by a type I pattern. 
  The multi-valued map principle implies 
  \begin{align}\label{eq.Asum}
    |\tilde A_{m+2}|=\sum_{\gamma\in \tilde A_m}{\frac{T_I(\gamma)}{T_{II}(\gamma)+1}}.
  \end{align}
  Similarly, by considering the multi-valued map from $\tilde A_{m+4}$ into $\tilde A_m$ 
  that replaces two type II patterns by type I patterns, one obtains
  $$|\tilde A_{m+4}|=\sum_{\gamma\in \tilde A_m}{\frac{T_I(\gamma)(T_I(\gamma)-1)}{(T_{II}(\gamma)+1)(T_{II}(\gamma)+2)}}.$$
  It follows that
  \begin{align*}
    & \frac{|\tilde A_{m+2}|^2}{|\tilde A_m|} -
    |\tilde A_{m+4}|\\
    &\qquad=
    \left( 
      \sum_{\gamma\in \tilde A_m}1
    \right)^{-1}
    \left(
      \sum_{\gamma\in \tilde A_m}
      \frac{T_I(\gamma)}
      {T_{II}(\gamma)+1}
    \right)^2 -
    \sum_{\gamma\in \tilde A_m}
    \frac{T_I(\gamma)(T_I(\gamma)-1)}{(T_{II}(\gamma)+1)(T_{II}(\gamma)+2)}
    \\
    &\qquad\le \sum_{\gamma\in \tilde A_m}
    \left( \left(\frac{T_I(\gamma)}
        {T_{II}(\gamma)+1}\right)^2 -
      \frac{T_I(\gamma)(T_I(\gamma)-1)}{(T_{II}(\gamma)+1)(T_{II}(\gamma)+2)} \right)
    \le \frac{2}{\delta^3 n}|\tilde A_m|.
  \end{align*}
  The first inequality is due to Cauchy-Schwarz and the second, valid when $n$ is high enough, 
  to $T_{II}(\gamma) \geq \de n + \ell \geq \de n - n^{3/4}$ 
  and $T_I(\gamma) \leq m \leq n + 2 n^{3/4}$. 
  Dividing the above by $|\tilde{A}_{m+2}|$ yields~\eqref{eq:crucial}.
  \medbreak

  Let us now show that
  \begin{align}\label{eq.Afrac}
    \frac{|\tilde A_{n - 2k + 2}|}{|\tilde A_{n - 2k}|} < \mu_c^2+\frac{1}{n^{1/5}}, \qquad \text{ for all } 0 \leq k \le n^{1/5}.
  \end{align}
  Assume instead that for some $0 \leq k \le n^{1/5}$ and $m = n - 2k$, 
  \begin{align*}
    \frac{|\tilde A_{m + 2}|}{|\tilde A_{m}|} \ge \mu_c^2+\frac{1}{n^{1/5}}.
  \end{align*}
  In particular, this implies that $\frac{|\tilde A_{m}|}{|\tilde A_{m+2}|} \le 1$. 
  Using~\eqref{eq:crucial}, it may be shown by recurrence that, for~$n$ large enough and $\ell\le n^{3/4}+k$, 
  we have~$\frac{|\tilde A_{m+2\ell}|}{|\tilde A_{m+2\ell+2}|} \le 1$ and
  \begin{align*}
    \frac{|\tilde A_{m + 2\ell + 2}|}{|\tilde A_{m +2\ell}|}
    \ge \mu_c^2+\frac{1}{n^{1/5}}- \frac{2}{\delta^3n}\sum_{k=1}^{\ell}{\frac{|\tilde{A}_{m+2k-2}|}{|\tilde{A}_{m+2k}|}}
    \ge \mu_c^2+\frac{1}{2n^{1/5}}.
  \end{align*}
  Thus, 
  \begin{align*}
    |\tilde A_{n+2n^{3/4}}| 
    &\ge |\tilde A_n| \left(\mu_c^2+\frac{1}{2n^{1/5}}\right)^{n^{3/4}} \\
    &\ge \frac12e^{-c\sqrt n}\mu_c^n \left(\mu_c^2+\frac{1}{2n^{1/5}}\right)^{n^{3/4}} \\
    & > e^{c_{HW}\sqrt{n+2n^{3/4}}}\mu_c^{n+2n^{3/4}}.
  \end{align*}
  In the second inequality, we used \eqref{eq:ab},
  and, in the third, we assumed that  $n$ exceeds an integer that is determined by $c$, $c_{HW}$ and $\mu_c$. 
  The above contradicts the Hammersley-Welsh bound \eqref{eq.hw}, and \eqref{eq.Afrac} is proved.
  \medbreak 
  
We conclude by observing that \eqref{eq.Afrac} and \eqref{eq:ab} imply that, for all $k \leq n^{1/5}$, 
  $$\tfrac12|A_n|\le |\tilde A_n|\le \left(\mu_c^2+\frac{1}{n^{1/5}}\right)^k|\tilde A_{n-2k}|\le C \mu_c^{2k}|\tilde A_{n-2k}|\le C\mu_c^{2k}|A_{n-2k}|,$$
  where $C$ is some constant depending only on $\mu_c$. 
\end{proof}
 
\begin{remark}
  If one assumes that $|A_m|\ge e^{-c\sqrt m}\mu_c^m$ for any $m \in [n-n^{3/4}, n + n^{3/4}]$, then the same technique implies a stronger result.
  In addition to \eqref{eq.Afrac}, a converse inequality may be obtained by a similar argument. 
  Assuming $|A_{m+2}|/|A_{m}|\le \mu_c^2-\frac{1}{n^{1/5}}$ for some $m \in [n-2n^{1/5}, n + 2n^{1/5}]$
  leads to a contradiction by going backward instead of forward. 
  It follows that there exists a constant $C > 0$ such that the ratio 
  $\mu_c^{2k}|A_n|/|A_{n+2k}|$ is contained between $1/C$ and $C$ for all $|k|\le n^{1/5}$. 
  We expect that in most applications, the important bound will be the one given by Lemma~\ref{lem:stability}.  
\end{remark}

\subsection{Redistribution of patterns}

We present a technical result, Lemma~\ref{lem:balanced}, concerning the distribution of patterns of type I and II within a given typical shell. Roughly speaking,
when considering walks of a given length with a given shell, there is a specified number of type I patterns to be allocated into the available slots, and this allocation occurs uniformly. For a typical shell, the number of slots, and the number of type I patterns to be allocated into them, are macroscopic quantities, of the order of the walk's length. 
Thus, conditionally on a typical shell, the number of type I patterns in a macroscopic part of the walk has a Gaussian behaviour, with variance of the order of the walk's length.
 
Lemma~\ref{lem:balanced} will prove to be very useful: after its proof, we will derive Proposition~\ref{prop:midpoint}, our result concerning midpoint delocalization, as a corollary. 
The lemma will also play an important role in our quantitative study of endpoint delocalization in Section~\ref{sec:quantitative}.

Consider a shell  $\sigma$ and $(S_1,S_2)$ a partition of its slots.
For a walk $\ga \in \sigma$ and $i=1,2$, let $T^i_I(\ga)$ (and $T^i_{II}(\ga)$) be the number of type I (and type II) patterns in $S_i$. 
With this notation, $T_I=T^1_I+T^2_I$ and $T_{II}=T^1_{II}+T^2_{II}$. 

\begin{lemma}\label{lem:balanced}
  Let $\delta, \eps, \eps', C>0$. 
  There exists $N>0$ such that the following occurs. 
  Let $n \geq N$, $\sigma$ be a shell and $(S_1,S_2)$ be a partition of its slots with $|S_1|,|S_2|\ge \delta n$. 
  Suppose that $n$ and $\sigma$ are such that $T_I$ and $T_{II}$ are both larger than $\delta n$  
  (and recall that $T_I$ and $T_{II}$ are determined by $n$ and $\sigma$).
  Then,
  \begin{align}\label{eq.balanced}  
    \PSAW_n\left(\left| T^1_I(\Ga)-\frac{T_I |S_1|}{|S_1|+|S_2|}\right| 
      \ge \sqrt n (\log n)^{1/2 + \eps} \mcond \varsigma(\Gamma)=\sigma\right)
    \leq \frac{1}{n^C}.
  \end{align}
  Moreover, if $k_1, k_2$ are such that 
  $\left| k_i-\frac{T_I |S_1|}{|S_1|+|S_2|}\right|  \leq 2 \sqrt{n} (\log n)^{1/2 + \eps}$ 
  and $| k_1 - k_2 | \leq \sqrt{n}$, then
  \begin{align}\label{eq.resample}
    \frac
    {\PSAW_n\left( T^1_I(\Ga) = k_1 \mcond \varsigma(\Gamma)=\sigma \right) }
    {\PSAW_n\left( T^1_I(\Ga) = k_2 \mcond \varsigma(\Gamma)=\sigma \right) }
    \geq n^{-\eps'}.
  \end{align}
\end{lemma}

Before the technical proof, we give a heuristic explanation. 
For simplicity, suppose that $S_1$ comprises the first $K$ slots of $\sigma$, and $S_2$ the remainder. 

Consider the random process $W: \{ 0,\ldots, T_I + T_{II} \} \to \{0,\ldots,T_{I}\}$ 
whose value at $0 \leq k \leq T_I + T_{II}$ is the number of type I patterns allocated into the first $k$ slots of $\sigma$. 
Under $\PSAW_n \left(\cdot| \varsigma(\Ga) = \sigma \right)$, 
$W$ is uniform in the set of trajectories of length $T_{I} + T_{II}$, with steps $0$ or $1$, starting at~$0$ and ending at~$T_{I}$.

In other words, consider the random walk with increments zero with probability $\tfrac{T_{II}}{T_I + T_{II}}$ 
and one with probability $\tfrac{T_{I}}{T_I + T_{II}}$.
Then $W$ has the law of this walk, conditioned on arriving at~$T_{I}$ at time $T_I + T_{II}$.
The assumptions on~$T_I$ and $T_{II}$ ensure that the variance of the increment of $W$ is bounded away from $0$. 

With this notation $T_I^1 = W_{K}$. 
The inequalities for $|S_1|$ and $|S_2|$ ensure that the point $K$ is not too close to the endpoints of the range of $W$. 
It follows from standard estimates on random walk bridges that $\frac{T_I^1}{\sqrt{T_I + T_{II}}}$ follows an approximately Gaussian distribution. 
If this approximation is used, then Lemma \ref{lem:balanced} follows by basic computations. 
Also, the probability that $W_K$ equals $\ell$ is at most $C n^{-1/2}$ for some constant $C > 0$ and any $\ell \in \ZZ$.
This last observation will be used in the proof of Proposition \ref{prop:midpoint}.

Finally we mention that a similar idea was used in~\cite{JROSTW} to obtain~$n^{-1/2}$~variations for the writhe
of self-avoiding polygons of length $n$.  

\begin{proof}[Proof of Lemma \ref{lem:balanced}]
  Fix $\delta, \eps, \eps'$ and $C$ strictly positive. 
  Let $n$, $\sigma$ and $(S_1, S_2)$ be as in the lemma. 
  The parameter $n$ will be assumed to be large in the sense that $n \geq N$ for some $N = N(\delta, \eps, \eps', C)$. 
  
  If $\Ga$ is distributed according to $\PSAW_n(. | \varsigma(\ga) = \sigma)$, then 
  the $T_I$ type I patterns and $T_{II}$ type II patterns are distributed uniformly in the slots of $\sigma$. 
  Thus, for $k \in \{0, \dots, |S_1| \}$,
  \begin{align}\label{eq:Tdistrib}
    \PSAW_n\left(T^1_I(\Ga)=k \mcond \varsigma(\Gamma)=\sigma\right) 
    = \frac{\binom{|S_1|}{k} \binom{|S_2|}{T_I - k}}{\binom{|S_1|+|S_2|}{T_I}}.
  \end{align}

  Write $m = |S_1| + |S_2|$, $|S_1| = \alpha m$ and   $T_I = \beta m$.
  By assumption $\alpha, \beta \in [\de , 1-\de]$ and $m \geq 2\de n$. 
  Let $Z = \frac{T^1_I}{\alpha \beta m} -1$. 
  Under $\PSAW_n\left( \cdot \mcond \varsigma(\Gamma)=\sigma\right)$, 
  $Z$ is a random variable of mean $0$,
  such that $\alpha \beta (1 + Z) m \in \ZZ \cap [0, \min \{ |S_1|,T_I \} ]$.
  
  First, we investigate the case where $Z$ is close to its mean, corresponding to the second part of the lemma.
  By means of a computation which uses Stirling's approximation and the explicit formula \eqref{eq:Tdistrib},  we find that
  \begin{align}\label{eq.zdistrib}
    \PSAW_n\left(Z = z \mcond \varsigma(\Gamma)=\sigma\right) 
    = (1 + o(1))
    \frac{ \exp \left(- \frac{\alpha \beta}{2(1-\alpha)(1-\beta)} m z^2 \right)}
    {\sqrt{2 \pi \alpha \beta (1 - \alpha)(1 - \beta) m}}, 
  \end{align}
  where $o(1)$ designates a quantity tending to $0$ as $n$ tends to infinity, 
  uniformly in the acceptable choices of $\sigma$, $S_1$, $S_2$ and $z$, 
  with $|z| \leq \frac{2 \sqrt{n}(\log n)^{1/2 +\eps}}{\alpha \beta m}$.

  Consider now $k_1, k_2$ be as in the second part of the lemma.
  Define the corresponding $z_i = \frac{k_i}{\alpha \beta m} - 1$,
  and note that $|z_i| \leq \frac{2\sqrt{n}(\log n)^{1/2 +\eps}}{\alpha \beta m}$
  and $|z_1  - z_2 | \leq \frac{\sqrt n }{\alpha \beta m}$. 
  By \eqref{eq.zdistrib},
  \begin{align*}
    \frac{ \PSAW_n\left(Z = z_1 \mcond \varsigma(\Gamma)=\sigma\right) }{\PSAW_n\left(Z = z_2 \mcond \varsigma(\Gamma)=\sigma\right) } 
    =  (1+o(1)) \exp \left( - \frac{1}{2 \de^5} (\log n)^{1/2 +\eps} \right)  
    \geq n^{-\eps'},
  \end{align*}
  for $n$ large enough. This proves \eqref{eq.resample}.

  We now turn to the deviations of $Z$ from its mean (this corresponds to the first part of the lemma). 
  From \eqref{eq:Tdistrib}, one can easily derive that
  $\PSAW_n\left(Z = z \mcond \varsigma(\Gamma)=\sigma\right)$ 
  is unimodal in $z$ with maximum at the value closest to $0$ that $Z$ may take
  (we remind the reader that $Z$ takes values in $\frac{1}{\alpha \beta  m } \Z - 1$, 
  which contains $0$ only if ${\alpha \beta  m } \in \Z$) .  
  For $ |z| \geq \frac{\sqrt{n}(\log n)^{1/2 +\eps}}{\alpha \beta m}$,
\eqref{eq.zdistrib} implies the existence of constants $c_0, c_1 > 0$ depending only on $\de$ such that, for $n$ large enough, 
  \begin{align}\label{eq.balanced0}
    \PSAW_n\left(Z = z \mcond \varsigma(\Gamma)=\sigma\right) \leq
    \frac{\exp\left( - c_0 ( \log n)^{1+ 2\eps}\right)}{c_1  \sqrt{n}  }
    \leq n^{- C - 1}.
  \end{align}
  Since $T^1_I$ takes no more than $n$ values, \eqref{eq.balanced0} implies \eqref{eq.balanced}.
\end{proof}

We are now in a position to prove Proposition~\ref{prop:midpoint}.
\begin{proof}[Proof of Proposition \ref{prop:midpoint}]
  It suffices to prove the statement for $n$ large. 
  Let $n \in \NN$ and $x \in \ZZ^d$.

  Consider a shell $\si$ with $x \in \ga$ for some walk $\ga \in \si$. 
  Let $S_1$ be the slots of $\si$ before $x$ and let $S_2$ be the ones after. 
  We will omit here the case where $x$ is a point contained in one of the slots;
  this is purely a technical issue and does not change the proof in any significant way. 
  
  Consider walks $\ga \in \SAW_n$ with $\varsigma (\ga) = \si$.
  Let $t_\si$ be the number of type~I patterns that such a walk needs to have in $S_1$ so that $\ga_{\lfloor \frac{n}{2} \rfloor} = x$, 
  if such a number exists and is contained in $[0, \min \{ |S_1|, T_I \} ]$.
  Denote $\sS$ the set of shells for which $t_\si$ is well defined. 
  Thus, if $\ga$ has midpoint $x$, then $\varsigma(\ga) \in \sS$. 
  We may therefore write
  \begin{align}\label{eq.shell_decomp}
    \PSAW_n\left(\Ga_{\lfloor \frac{n}{2} \rfloor} = x \right) = 
    \sum_{\si \in \sS}  \PSAW_n\left( T^1_I = t_\si , \varsigma(\Ga) = \si \right).
  \end{align}
  There exist constants $\de ,c > 0$ such that 
  \begin{align}\label{eq.patterns2}
    &\PSAW_n \left( \Ga[0, \tfrac{n}{4}] \text{ contains fewer than $\de n$ type I patterns} \right) \\
    &\le  \frac{c_{n / 4} c_{3n / 4}}{c_n} \PSAW_{n/4} \left( \Ga\text{ contains fewer than $\de n$ type I patterns}\right) \nonumber \\
    &\le e^{2c_{HW}\sqrt n}\PSAW_{n/4} \left( \Ga\text{ contains fewer than $\de n$ type I patterns} \right) \nonumber
    < e^{-c n},
  \end{align}
  where the second inequality comes from the Hammersley-Welsh bound \eqref{eq.hw}, and the third from \eqref{eq.patterns}.
  The same holds for type II patterns, and for $\Ga[\tfrac{3n}{4}, n]$ instead of $\Ga[0, \tfrac{n}{4}]$.
  It follows that 
  \begin{align}\label{eq.smallS}
    \PSAW_n \big( \min(|S_1|,|S_2|,T_I,T_{II}) < \de n \big) < 4e^{-c n}.
  \end{align}
  For shells $\si \in \sS$ such that $|S_1| \geq \de n$ and $|S_2| \geq \de n$, 
   \eqref{eq.zdistrib} gives
  \begin{align}\label{eq.Tuncertainty}
    \PSAW_n\left( T^1_I = k \mcond \varsigma(\Ga) = \si \right) \leq C n^{- 1/2},
  \end{align}
  for some $C > 0$, any $k \in \NN$ and $n$ sufficiently large (i.e., larger than some value depending only on $\de$). 

  By applying \eqref{eq.smallS} and \eqref{eq.Tuncertainty} to \eqref{eq.shell_decomp}, we obtain 
  \begin{align*}
    \PSAW_n\left(\Ga_{\lfloor \frac{n}{2} \rfloor} = x \right) 
    \leq 4e^{-c n} + C n^{- 1/2} 
    \leq 2C n^{- 1/2},
  \end{align*}
  for $n$ sufficiently large. 
\end{proof}

\section{Delocalization of the endpoint}\label{sec:uniform}

This section is devoted to the proof of Theorem~\ref{theorem:SAW_to_x}. Let us begin with some general definitions. 
A walk $\gamma\in \sawset_n$ is called a {\em bridge} if  
\begin{align*}
  \langle\g_0| e_1\rangle < \langle\g_k| e_1\rangle \le \langle\g_{n}| e_1\rangle , 
  \quad \text{ for $ 0 < k \le n $}.
\end{align*}
Write $\sabset_n$ for the set of bridges of length $n$ 
and $\sab_n$ for the uniform measure on $\sabset_n$.

\begin{figure}[htb]
  \begin{center}
    \cpsfrag{ren}{$\ga_k$}
    \cpsfrag{zren}{$\ga_\ell$}
    \cpsfrag{hang}{$\ga_\hang$}
    \cpsfrag{start}{$O$}
    \cpsfrag{end}{$\ga_n$}
    \includegraphics[width=0.25\textwidth]{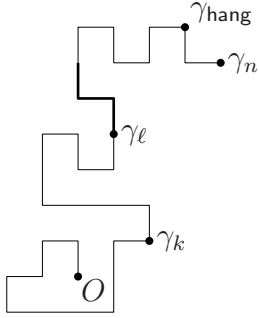}
  \end{center}
  \caption{A walk $\ga$ with a renewal point $\ga_k$, 
    a $z$-renewal point $\ga_\ell$ and hanging point $\ga_\hang$.
    The bold structure beyond the point $\ga_\ell$ helps to ensure that $\ell$ is a $z$-renewal time.
    Both $\ga[k, \ell]$ and $\ga[k, \hang]$  are bridges.}
  \label{fig:zren}
\end{figure}

For $\ga \in \sawset_n$, an index $k \in [0,n]$ is a {\em renewal time} 
if $\langle\gamma_i | e_1\rangle \le \langle\gamma_{k} | e_1\rangle$ for  $0 \leq i < k$ and
$\langle\gamma_i | e_1\rangle > \langle\gamma_{k} | e_1\rangle$ for  $n \geq i > k$. 
Because it simplifies the proof of the next subsection in a substantial way, we introduce the notion of $z$-renewal. An index $k \in [0,n-2]$ is a {\em $z$-renewal time} if 
\begin{itemize}
\item $\langle\gamma_i | e_1\rangle < \langle\gamma_{k+1} | e_1\rangle$ for  $0 \leq i < k+1$,
\item $\langle\gamma_{k+1} | e_1\rangle = \langle\gamma_{k+2} | e_1\rangle$,
\item  $\langle\gamma_i | e_1\rangle > \langle\gamma_{k+1} | e_1\rangle$ for  $n \geq i > k+2$. 
\end{itemize}
Note that a $z$-renewal time is necessarily a renewal time. 
Let $\ren_\gamma$ denote the set of $z$-renewal times of $\gamma$. A ($z$-)renewal point is a point of the form $\gamma_k$ where $k$ is a ($z$-)renewal time.

\subsection{The case of bridges} 

Let $\pi_1$ be the orthogonal projection from $\ZZ^d$ onto the hyperplane 
$\mathbb H=\{x\in \ZZ^d:\langle x | e_1\rangle = 0\}$, 
that is,  for any $x\in \ZZ^d$,
$$ \pi_1(x) = x - \langle x| e_1\rangle e_1.$$

Our first step on the route to Theorem~\ref{theorem:SAW_to_x} is its analogue for bridges. 
Bridges are easier to handle due to their renewal and $z$-renewal points.

\begin{proposition}\label{lemma:SAB_to_x}
We have that
  $\lim_n \sup_{v \in \mathbb{H}} \sab_n\left( \pi_1(\Gamma_n) = v \right) = 0$.
\end{proposition}

This proposition follows from the next two statements. 
The first shows that typical bridges have many $z$-renewal times, and 
the second that the endpoint of a bridge with many $z$-renewal times is delocalized. 
\begin{proposition}\label{prop:M}
  For any $M\in \NN$, $\lim_n \sab_n \left( |\ren_\Gamma|< M \right) = 0$.
\end{proposition}

\begin{proof}
  Fix $\eps > 0$ and $M \in \NN$ and let us show that, for $n$ large enough, $\sab_n \left( |\ren_\Gamma|< M \right) < \eps$. 

  Let $\SAB^M_n$ be the set of bridges of length $n$ with strictly fewer than $M$ z-renewal times. 
  If $|\SAB^M_n| <  e^{-2c_{HW}\sqrt n}\mu_c^n$, we may use \eqref{eq.hw} to deduce that 
  $\sab_n \left( |\ren_\Gamma|< M \right) < e^{-c_{HW}\sqrt n} \leq \eps$, provided $n$ is large enough. 
  
  From now on, assume 
  \begin{equation}\label{eq: not too small}
    |\SAB^M_n|\ge e^{-2c_{HW} \sqrt n}\mu_c^n.
  \end{equation}
  Let $k =  \lfloor n^{1/5} \rfloor$, and define the map 
  $$ \Phi:\bigcup_{j=1}^{k}\SAB_{n-2j}^M\times\SAB_{2j-2}\longrightarrow\SAB_{n}$$ 
  that maps $(\ga_1,\ga_2)$ to the concatenation of $\ga_1$, the walk whose consecutive edges are $e_1$ and $e_2$, and $\ga_2$. 
  Each $\ga  \in \SAB_{n}$ has at most $M$ pre-images under $\Phi$, 
  because, if $\Phi$ maps $(\gamma_1,\gamma_2)$ to $\gamma$, 
  then the endpoint of the copy of $\gamma_1$ in $\gamma$ is one of the first $M$ 
  $z$-renewal points of $\gamma$. 
  We deduce that
  $$\sum_{j=1}^k|\SAB_{n-2j}^M|\times|\SAB_{2j-2}|\le M |\SAB_{n}|.$$
  Lemma~\ref{lem:stability} applied with $A = \cup_{k\geq 0} \SAB_{k}^M$, along with \eqref{eq: not too small},
  provides the existence of a constant $C > 0$ for which  
  $|\SAB_{n-2j}^M|\ge C \mu_c^{-2j}|\SAB^M_n|$ for all $0 \le j \le k$. 
  Thus,
  $$\sab_n \left( |\ren_\Gamma|< M  \right) = \frac{|\SAB_n^M|}{|\SAB_n|}\le \frac{M}{C \sum_{j=1}^k|\SAB_{2j-2}|\mu_c^{-2j}}.$$
  However, it is a classical fact that the generating function for bridges diverges at criticality, 
  i.e. $\sum_{j=1}^\infty|\SAB_{2j}| \mu_c^{-2j}  = \infty$. This is shown for instance in \cite[Cor. 3.1.8]{MS93}.
  As $n$ tends to infinity, so does $k$, and therefore,  for $n$ sufficiently large, $\sab_n \left( |\ren_\Gamma|< M  \right) < \eps$. 
\end{proof}

\begin{proposition}\label{prop:Mvariation}
  For any $\ep>0$, there exists $M>0$ such that, for any $n,h>0$ and $x\in \mathbb Z^d$,  
  $$
  \sab_n \left( \Gamma_n=x \, \Big\vert \,  | \ren_\Gamma|\ge M,\langle \Gamma_n|e_1\rangle=h \right) \le \ep .
  $$
\end{proposition}

\begin{proof}[Proof]
  If $k$ is a $z$-renewal time of a walk $\gamma$, and if the edge $(\gamma_{k+1},\gamma_{k+2})$ is modified to take any one of the  $2d-2$ values $\pm e_2$, $\pm e_3$, $\ldots$, the outcome remains self-avoiding and shares $\gamma$'s height. In light of this, the proof is a routine exercise.
\end{proof}

\subsection{The case of half-space walks}

Next, we prove delocalization for walks confined to a half-space.
The set of {\em half-space walks} of length $n$ is
\begin{align*} 
  \sahswset_n &= \Big\{ \gamma \in \SAW_n: \langle\gamma_k |e_1\rangle > 0 \text{ for all } 1 \leq k \le n \Big\}.
\end{align*} 
Let $\sahsw_n$ denote the uniform measure on $\sahswset_n$.

\begin{proposition}\label{prop:half-space}
  We have that $\lim_n \sup_{x \in \ZZ^d} \sahsw_n(\gamma_n = x)  = 0$. 
\end{proposition}

\begin{proof}
  Let $\ep>0$ and note that  Proposition~\ref{lemma:SAB_to_x} ensures the existence of $H \in \N$ such that
  \begin{equation}\label{eq:bound}
    \sup_{k\ge H,v\in\mathbb H}\sab_k\big( \pi_1(\Gamma_k)=v \big) \le \ep.
  \end{equation}
 
  First note that $\sahsw_n ( \langle \gamma_\hang|e_1\rangle\le H)$ 
  decays exponentially as $n \to \infty$.
  Indeed, the pattern theorem \cite[Thm. 1]{kestenone} implies that, 
  with probability exponentially close to one, a walk in $\sahswset_n$ contains $H+1$ consecutive edges $e_1$. 
  We may therefore restrict our attention to walks going above height~$H$.

  \begin{figure}
    \begin{center}
      \cpsfrag{hang}{$\ga_\hang$}
      \cpsfrag{ren}{$\unf(\gamma)_\re$}
      \cpsfrag{bren}{$b_\re$}
      \cpsfrag{bn}{$b_{n+1}$}
      \cpsfrag{x}{$x$}
      \cpsfrag{O}{$O$}
      \cpsfrag{eone}{$e_1$}
      \includegraphics[width=0.8\textwidth]{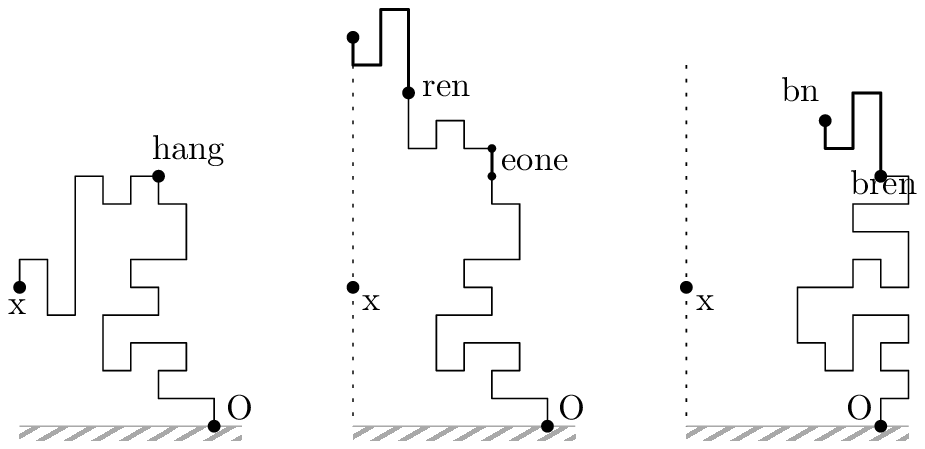}
    \end{center}
    \caption{\emph{Left:} A walk $\gamma\in \sahswset_n$ ending at $x$.
      \emph{Middle:} The unfolding of $\gamma$, $\unf(\ga)$, and its last renewal point $\unf(\gamma)_\re$.
      Notice the bold 
      edge $e_1$ added between $\ga^1$ and the reflection of $\ga^2$.
      \emph{Right:} A walk $b \in \Phi(\gamma)$. 
      Its last renewal point is $b_\re$;
      $b[0,\re]$ is a bridge and 
      $b[\re,n+1]$ is equal up to translation to  $\unf(\gamma)[\re,n+1]$ (bold).
      The choice of $b[0,\re]$ may be such that $\pi_1(b_{n+1}) \neq \pi_1(x)$.}
    \label{fig:half_sp}
  \end{figure}

  Let $x \in \ZZ^d$ and define a multi-valued map $\Phi : \{\gamma \in\sahswset_n:\gamma_n=x\} \to \sahswset_{n + 1}$ as follows.  
  Let $\re$ be the last renewal time of $\unf(\gamma)$ 
  (recall the definition of $\unf$ from Section~\ref{sec:unfold}) 
  and let $\Phi(\gamma)$ be the set of all half-space walks which can be represented as the concatenation 
  of some bridge of length $\re$ and~$\unf(\gamma)[\re,n]$.
  See Figure \ref{fig:half_sp}.
  Note that $\hang$ is a renewal time for  $\unf(\gamma)$ (this is due to the edge $e_1$ added between the two walks),
  hence $\re$ is well defined and $\re \geq \hang$. 
  For any $\gamma\in\sahswset_n$, $|\Phi(\gamma)| =|\sabset_{\re}|$. 
 
  In the other direction, let $b \in \sahswset_{n+1}$ and $\ga \in \Phi^{-1}(b)$. 
  The time $\re$ of~$\ga$ can be determined, since it is the last renewal time of $b$. 
  As such, $\ga[\re, n]$ is determined by $b$. 
  Thus, $\ga[0,\re]$ is a bridge with 
  $$\pi_1(\ga_\re)=\pi_1(x+b_\re-b_{n+1}).$$ 
  Furthermore, $\unf$ is an injective function from $\{\chi \in \sahswset_n:\chi_n=x\}$ to $\sahswset_{n+1}$
  (indeed, the vertical coordinate of the hanging point of the original walk can be determined from knowing that the original walk ended at $x$). 
  In conclusion, 
  $$
  |\Phi^{-1}(b)| \leq \Big| \Big\{\chi \in \sabset_{\re} : \pi_1(\chi_\re)=\pi_1( x + b_\re-b_{n+1})   \Big\} \Big|.
  $$
  As mentioned before, we may suppose $\langle \ga_{\hang} | e_1 \rangle \ge H$, and therefore $\re \ge H$. 
  The set $\Phi(\gamma)$ is independent of $\gamma\in \Phi^{-1}(b)$. 
  Thus, for any choice of $b$, the contracting factor of~$\Phi$ appearing in the multi-valued principle satisfies
  $$ \Lambda_\Phi(b) \leq \max_{k\ge H,v\in\mathbb H}\sab_k \big( \pi_1(\Gamma_k)=v \big) \le \ep. $$
  The multi-valued map principle and the trivial inequality $|\sahswset_{n+1}| \leq 2d|\sahswset_n|$ yield
  $\sahsw_n(\gamma_n = x) \leq 2d \eps$, and the proof is complete. 
\end{proof} 
 
\subsection{The case of walks (proof of Theorem \ref{theorem:SAW_to_x})}

Fix $\ep>0$.
Proposition~\ref{prop:Mvariation} yields the existence of $M>0$ such that, for any $n, h \geq 0$, 
\begin{equation}\label{defM} 
\sup_{x \in \ZZ^d} \sab_n \left( \Gamma_n = x \, \Big\vert \,  |\ren_\Gamma| \ge M, \langle \Gamma_n| e_1\rangle=h \right) \le \frac{\ep}{2d}. 
\end{equation}
Fix such a value of $M$.
Recall the notation $(\ga^1, \ga^2) = (\ga[0, \hang], \ga[\hang, n])$. 
We divide the proof in two cases, depending on whether $\Gamma^1$ possesses at least or fewer than $M$ $z$-renewal times. 
The next two lemmas treat the two cases.

\begin{lemma}[Many $z$-renewal times for $\Gamma^1$]\label{lem:many R}
  For  $n$ large enough, 
  $$
  \sup_{x \in \ZZ^d} \saw_n \big( \Gamma_n=x\text{ and } | \ren_{\Gamma^1}| \geq M \big) \le \ep.
  $$
\end{lemma}

In the case where the first part contains few $z$-renewal times, 
we further demand that the hanging time be smaller than $n/2$. 

\begin{lemma}[Few $z$-renewal times for $\Gamma^1$]\label{lem:few R}
  For  $n$ large enough, 
  $$
  \sup_{x \in \ZZ^d} \saw_n \big( \Gamma_n=x, \hang\le n/2\text{ and }|\ren_{\Gamma^1}| < M \big) \le \ep.
  $$
\end{lemma}

Theorem~\ref{theorem:SAW_to_x} follows from the two lemmas because, as we now see, 
they imply that
$$\sup_{x \in \ZZ^d}\saw_n (\Gamma_n=x) \le 4\ep.$$
Indeed, the lemmas clearly yield
$$\sup_{x \in \ZZ^d} \saw_n \big( \hang\le n/2\text{ and }\Gamma_n=x \big) \le 2\ep,$$
for $n$ large enough. The counterpart inequality with $\hang \ge n/2$ may be obtained by reversing the walk's orientation 
(and translating it to start at the origin). 

Thus the proof of Theorem~\ref{theorem:SAW_to_x} is reduced to demonstrating the two lemmas.

\begin{proof}[Proof of Lemma~\ref{lem:many R}]  
  For $n \in \NN$, let $E_M(x)$ be the set of walks  $\ga \in \sawset_n$ with $\ga_n=x$ and $|\ren_{\ga^1}| \geq M$.
  For such walks, let $\re$ be the smallest renewal time of $\ga^1$. 
  Split~$\ga^1$ in two parts   $\gamma^{11} = \g[0,\re]$, $\gamma^{12} = \ga[\re,\hang]$ and set $k=\hang - \re$.

  \begin{figure}
    \begin{center}
      \cpsfrag{hang}{$\ga_\hang$}
      \cpsfrag{ren}{$\ga_\re$}
      \cpsfrag{bhang}{$b_{\hang(\ga)}$}
      \cpsfrag{bren}{$b_\re$}
      \cpsfrag{bn}{$b_{n+1}$}
      \cpsfrag{target}{$x$}
      \cpsfrag{o}{$O$}
      \includegraphics[width=0.7\textwidth]{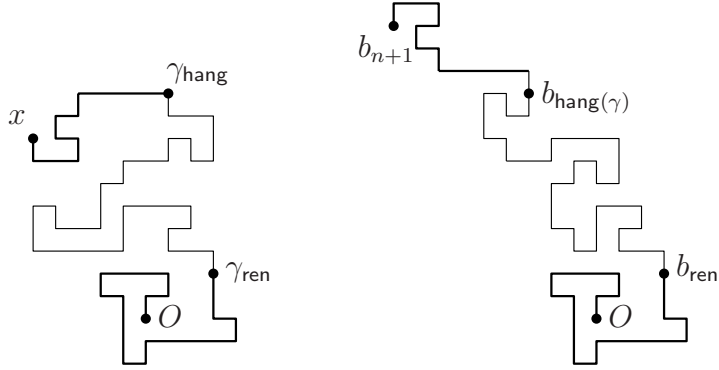}
    \end{center}
    \caption{\emph{Left:} A walk $\gamma\in E_M(x)$.
      \emph{Right:} A walk $b \in \Phi(\ga)$. 
      The first and last parts of $b$ (bold) 
      are the same as those of $\ga$ (up to translation and reflection).
      Its middle part is a bridge with appropriate length and $e_1$-displacement.}
    \label{fig:many_R}
  \end{figure}
   
  We define a multi-valued map $\Phi : E_M(x) \to \SAW_{n+1}$ under which
  $\gamma \in E_M(x)$ is unfolded about $\hang$ and the sub-bridge $\gamma^{12}$ 
  is substituted by any bridge sharing $\gamma^{12}$'s length and $e_1$-displacement 
  (but not necessarily its displacement in other directions) and having at least~$M$ $z$-renewal points. 
  More precisely, for $\gamma \in E_M(x)$, $\Phi(\gamma)$ is the set of walks $b \in \sawset_{n+1}$ with the properties that
  \begin{itemize*}
  \item $b[0,\re]$ is equal to $\gamma^{11}$;
  \item $b[\re,\hang]$ is the translate of a bridge of length $k$ and $e_1$-displacement 
    $$
    \langle \gamma_\hang-\g_\re|e_1\rangle , 
    $$ 
    and which has at least $M$ $z$-renewal times;
  \item and $b[\hang, n+1]$ is equal to $\unf(\gamma^2)$ (up to translation).
  \end{itemize*}
  By construction such walks are indeed self-avoiding, so that $\Phi(\g)$ is well-defined. 
  See Figure \ref{fig:many_R}. 

  Let us now estimate the contracting factor $\Lambda_\Phi$ of $\Phi$. 
  For $\gamma \in E_M(x)$, 
  \begin{equation}\label{forward1}
    |\Phi(\gamma)| = \Big| \Big\{ \chi \in \sabset_k:|\ren_\chi| \ge M  \text{ and } 
    \langle \chi_k | e_1 \rangle = \langle \ga_\hang-\ga_\re |e_1 \rangle \Big\} \Big|.
  \end{equation}

  For the number of pre-images, consider $\gamma \in  \Phi^{-1}(b)$ for some $b\in \sawset_{n+1}$ 
  (note that $\Phi^{-1}(b)$ could be empty, in which case the conclusion is trivial).
  Since $\gamma$ ends at $x$ 
  and the $e_1$-displacement of the bridge which replaces $\gamma^{12}$ in $b$ is the same as that  of $\gamma^{12}$, 
  the $e_1$-coordinate of the hanging point of $\gamma$ is determined by $b$. Namely,
  \begin{align*}
    \langle\gamma_\hang| e_1\rangle 
    = \langle b_{\hang(\ga)} | e_1 \rangle
    = \frac{\langle x| e_1\rangle + \langle b_{n+1}| e_1\rangle - 1}{2} .
  \end{align*}
  But $\hang$ is a renewal time for $b$, hence the above determines $\hang$.
  It follows that $\gamma^2$ is also determined by $b$ (including its positioning which is given by the fact that $\ga_n = x$). 
  Moreover, since $\re$ is the first renewal time of $\ga$, it is also the first renewal time of $b[0, \hang]$.
  Thus $b$ determines $\ga^{11}$ as well. 
  Finally, $\ga^{12}$ is a bridge with at least $M$ $z$-renewals, between the determined points $\ga_\re$ and $\ga_\hang$.  
  It follows that 
  \begin{equation}\label{back1}
    |\Phi^{-1}(b)| \leq \Big| \Big\{ \chi \in\sabset_k: | \ren_\chi |\ge M \text{ and } \chi_k = \ga_\hang - \ga_\re \Big\} \Big| .
  \end{equation}
  Since $\gamma^{11}$ and $\gamma^2$ are determined by $b$, 
  any $\gamma\in \Phi^{-1}(b)$ has the same number of images under $\Phi$.  
  Equations \eqref{forward1}, \eqref{back1} and the choice of $M$ (see \eqref{defM}) imply
  that $\Lambda_\Phi(b)$ is bounded by $\frac{\ep}{2d}$ uniformly in $\gamma$, 
  which immediately yields
  $$
  |E_M(x)|\le \frac{\varepsilon}{2d}|\sawset_{n+1}|\le \varepsilon |\sawset_n|.
  $$
\end{proof}

We finish with the easier proof of Lemma~\ref{lem:few R}.
\begin{proof}[Proof of Lemma~\ref{lem:few R}] 
  Let $F_j(x)$ be the set of walks $\gamma \in \SAW_n$ such that $\gamma_n=x$, $\hang\le n/2$ and $|\ren_{\gamma^1}|=j$.
  
  We construct once again a multi-valued map $\Phi$, this time from $F_j(x)$ to~$\SAW_{n+3}$. 
  For $\ga\in F_j(x)$, $\Phi(\ga)$ comprises the walks formed by 
  concatenating~$\gamma^1$,  
  the walk whose consecutive edges are $e_1$ and  $e_2$, and any half-space walk of length $n-\hang + 1$. 
  See Figure \ref{fig:few_R}.

  The number of images through $\Phi$ satisfies $|\Phi(\gamma)|=|\sahswset_{n - \hang + 1}|$. 
  To determine the number of pre-images, 
  note that, if $b \in \Phi(\ga)$, 
  then $b_{\hang(\ga)}$ is the $(j+1)$-st $z$-renewal point of $b$.   
  Also, $\gamma^2$ is contained in the half-space
  $\{y \in \ZZ^d:\langle y|e_1\rangle \leq \langle\gamma_\hang|e_1\rangle \}$ 
  and ends at the point $x$. 
  Such walks can easily be transformed into half-space walks of length $n - \hang + 1$ 
  by reflecting them and adding an edge $e_1$ at the beginning. 
  Note that the endpoint of such a walk is then determined by $\gamma^1$ and $x$.
  Using Proposition~\ref{prop:half-space}, we find that, for any $b \in \SAW_{n+3}$, the contracting factor of $\Phi$ satisfies
  $$
  \Lambda_\Phi(b)\le \sup_{\substack{k\ge n/2 + 1 \\ z\in\ZZ^d}}\sahsw_k \big( \Gamma_k = z \big) \le \frac{\ep}{(2d)^3M},
  $$
  provided that $n$ is large enough.
  By the multi-valued map principle and $|SAW_{n+3}| \le (2d)^3|SAW_{n}|$, 
  we obtain that $\saw_n\big(F_j(x)\big) \le \ep/M$. 
  By taking the union of the $F_j(x)$ over $j < M$, we obtain Lemma~\ref{lem:few R}.
 \end{proof}

  \begin{figure}
    \begin{center}
      \cpsfrag{hang}{$\ga_\hang$}
      \cpsfrag{ren}{$\ga_\re$}
      \cpsfrag{bhang}{$b_{\hang(\ga)}$}
      \cpsfrag{bren}{$b_\re$}
      \cpsfrag{bn}{$b_{n+3}$}
      \cpsfrag{target}{$x$}
      \cpsfrag{o}{$O$}
      \includegraphics[width=0.5\textwidth]{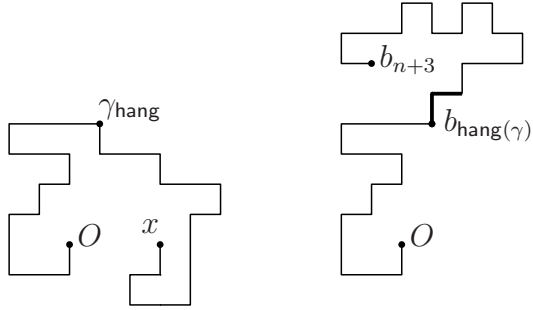}
    \end{center}
    \caption{\emph{Left:} A self-avoiding walk $\gamma\in F_j$.
      \emph{Right:} A walk $b \in \Phi(\ga)$. 
      The point $b_{\hang(\ga)}$ is the $(j+1)$-st z-renewal point of $b$; 
      it is followed by the edges $e_1$, $e_2$ (in bold), 
      then by a half-space walk of length $n-\hang + 1$.}
    \label{fig:few_R}
  \end{figure}

\section{Quantitative decay for the probability of ending at $x$}\label{sec:quantitative}

We say that a walk $\ga = \ga[0,n]$ \emph{closes} if $\g_0$ and $\g_n$ are neighbors; a closing walk is one that closes. 
Theorem \ref{t.nonclosing} claims that for any $\eps >0$ and $n$ large enough,
$\saw_n \big(  \Gamma \text{ closes} \big) \leq n^{-1/4 + \eps}.$

 
A little notation is in order as we prepare to prove Theorem~\ref{t.nonclosing}.
\begin{definition}
  Two closing walks are said to be equivalent if the sequence of vertices visited by one is a cyclic shift of this sequence for the other. 
  A (self-avoiding) polygon is an equivalence class for this equivalence relation. 
  The length of a polygon is equal to the length of any member closing walk plus one. 
  For $n \in \N$, let $\SAP_n$ be the set of polygons of length $n+1$. 
\end{definition}

The following trivial lemma will play an essential role.
\begin{lemma}\label{lem.polyinv}[Polygonal invariance]
  For $n \in \N$,  let $\chi$ and $\chi'$ be two equivalent length-$n$ closing walks. Then
  $$
  \PSAW_n\big(\Ga \, \textrm{is a translate of $\chi$} \big) = 
  \PSAW_n\big(\Ga   \, \textrm{is a translate of $\chi'$} \big).
  $$
\end{lemma}

\subsection{An overview of the proof of Theorem \ref{t.nonclosing}}\label{sec.heuristic}

We start by a non-rigorous overview of the proof of Theorem \ref{t.nonclosing}. 
The actual proof is in Subsection \ref{sec.actualproofclosing}.

The proof will proceed by contradiction. 
Suppose that the statement of Theorem \ref{t.nonclosing} is false, and let $n$ be a large integer such that 
$\saw_n \big(  \Gamma \text{ closes} \big) \geq n^{-1/4 + 5\eps}$.
The factor $5\eps$ in the exponent will be used as a margin of error which will decrease at several steps of the proof. 

Fix an index $\lo \in [\frac{n}{4}, \frac{3n}{4}]$ such that 
$\saw_n (  \Gamma \text{ closes} \cond \hang = \lo ) \geq n^{-1/4+4\ep}$. 
(The existence of such an index is proved in Lemma \ref{l.solid} and relies solely on polygonal invariance and the hypothesis that Theorem~\ref{t.nonclosing} is false.)

A walk  ending at its hanging point will be called \emph{good} if, 
when completed by $n - \lo$ steps in such a way that the hanging point is left unchanged,
the resulting walk has probability at least $n^{ - 1/4+3\ep}$ of closing. 
When thinking of walks as being built step by step, 
good walks should be thought of as first parts that leave a good chance for the walk to finally close. 
Since we assume that the walk closes with good probability, 
it is natural to expect that its first part is good with reasonable probability, 
and indeed one may prove (using polygonal invariance once again) that for our choice of $\lo$,
\begin{align*}
  \PSAW_n \left( \Ga^1 \text{ is good} \mcond \Ga \closes , \lexm = \lo \right) \geq n^{ - 1/4+3\ep}.
\end{align*}
Here, the notation $\Gamma^1 = \Gamma[0,\lexm]$ for the walk's first part was introduced in Definition~\ref{d.hang}. 

This estimate can be improved in the following way: one may change the value of the hanging time and prove that for $0 \leq k \leq \sqrt{n}$,
\begin{align}\label{eq.good2approx}
  \PSAW_n \left( \Ga^1 \text{ is good} \mcond \Ga \closes , \lexm = \lo - 2k \right) \geq n^{ - 1/4+2\ep}.
\end{align}
This part of the proof is heavily based on the resampling of patterns described in Lemma~\ref{lem:balanced} and illustrated in Figure~\ref{fig:pattern_redistrib}.

\begin{figure}[htb]
  \begin{center}
    \cpsfrag{hang}{$\ga_\hang$}
    \cpsfrag{galo}{$\ga_{\hang - \lo}$}
    \cpsfrag{end}{$\ga_n$}
    \includegraphics[width=0.95\textwidth]{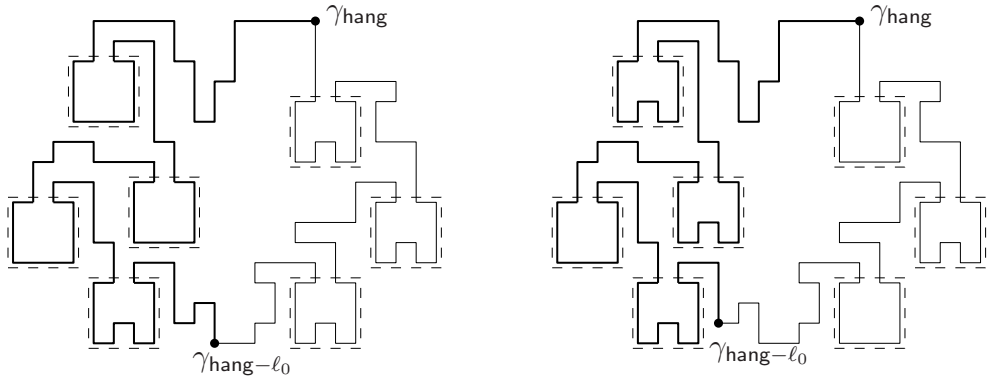}
  \end{center}
  \caption{In a closing walk the $\lo$ steps up to the hanging point form a good walk (bold). 
    By exchanging type I and II patterns between the bold and regular part of the walk in a way that increases by order $n^{1/2}$ the number of type $II$ patterns in the bold part, we may effectively shorten the good part by an amount of this order.}
  \label{fig:pattern_redistrib}
\end{figure}

This study shows that, when considering closing walks as polygons, the $\lo -2k$ steps before the hanging point 
have reasonable probability of forming a good walk, for $k = 0,\dots, \sqrt{n}$.
The correspondence between closing walks and polygons is essential here, 
as is the fact that the hanging point only depends on the polygon, not on the starting point of the closing walk. 

Call a point whose index lies in $[\lexm - \ell_0,\lexm]$ of a closing walk \emph{ticked} if the section of the walk between that point and the hanging point is good.  
Since each of $\sqrt{n}$ points have chance at least $n^{-1/4+2\ep}$ to be ticked,
the expected number of ticked points in a closing walk is at least $n^{1/4+2\ep}$. 
It follows that, with probability greater than $n^{-1/4+\ep}$, a closing walk has more than $n^{1/4+\ep}$ ticked points. 

We now reach the crucial part of the proof. Fix a walk  with $\sfT \geq n^{1/4+\ep}$ ticked points.
By considering the portions of the walk between the ticked points and the hanging point, 
we obtain a family of good walks $\{\chi^i: i = 1, \dots,  \sfT \}$, 
with $\chi^i \subset \chi^{i+1}$. See also Figure \ref{fig:avoided}.

The existence of this family of good walks implies a very strong property of $\chi^\sfT$. 
Indeed, let $\Ga$ be a uniform self-avoiding walk of length $n-\lo$, 
starting at the (common) end-point $z$ of the $\chi^i$s and with hanging point $z$ (in words, it stays in the half-space ``below'' $z$). Note three properties.
First, the events that $\Ga$ ends next to the starting point of $\chi^i$, for $ i =1, \dots,  \sfT$, 
are mutually exclusive (in fact, this is not quite true, as we will discuss in the proof).
Second, the events that $\Ga$ avoids $\chi^i$ are decreasing with $i$ (since $\chi^i$ is a portion of $\chi^{i+1}$). 
Third, note that $\chi^i$ being good means that, when conditioning $\Ga$ to avoid $\chi^i$, 
there is probability at least $n^{-1/4+3\ep}$ that $\chi^i$ ends next to the starting point of $\chi^i$. 
By using these three facts alongside $\sfT \geq n^{1/4+\ep}$, the probability that $\Ga$ avoids $\chi^\sfT$ can be proved to be stretched exponentially small, i.e. at most $e^{-cn^{4\eps}}$ for some small constant $c > 0$. 

Using an unfolding argument, this implies that, 
when conditioning on the~$\ell_0$ first steps of the walk to satisfy $\sfT \geq n^{1/4 + \eps}$ 
and resampling the end of the walk, the newly obtained walk has stretched exponentially small chance of having $\ell_0$ as its hanging time.
It is therefore also stretched exponentially unlikely for a walk to have $n^{1/4 + \eps}$ ticked points in its first $\ell_0$ steps and to have~$\ell_0$ as its hanging time.
 But by assumption, with probability $n^{-1/4 + 5\eps}$, a walk is closing; moreover, by polygonal invariance, it has conditional probability~$1/n$ to have $\ell_0$ as its hanging time; and finally, as we have discussed,  with a further conditional probability  of at least  $n^{-1/4 + \eps}$, its first $\ell_0$ steps have $\sfT \geq n^{1/4 + \eps}$. Thus, the above event is both of probability at most stretched exponential and of probability at least $n^{-3/2 + 6\eps}$, which of course is a contradiction if $n$ is large enough.

\subsection{Proof of Theorem~\ref{t.nonclosing}}\label{sec.actualproofclosing} 

We now elaborate the heuristic argument presented in Subsection~\ref{sec.heuristic}.
As mentioned before, we will proceed by contradiction. 
Suppose that there exists $\eps >0$ such that 
\begin{align}\label{eq.assumption}
  \PSAW_n \big( \Ga \closes \big) \geq n^{ - \frac{1}{4} + 5 \eps}
\end{align}
for an infinite number of values of $n \in \N$. (In particular $\eps \leq 1/20$.)

Fix $n \geq \max \{ 2, 4^{1/\eps} + 1 \}$ for which \eqref{eq.assumption} holds. 
Further bounds on $n$ (depending only on $\eps$) will be imposed. 
The next lemma and the bound $n > 4^{1/\eps}$ permit us to fix an integer $\lo \in [\frac{n}{4}, \frac{3n}{4}]$ such that 
\begin{equation}\label{eq:solidl}
\PSAW_n \big( \Ga \closes \, \big\vert \, \lexm = \lo \big) \geq n^{ - \frac{1}{4} + 4 \eps}.
\end{equation} 

\begin{lemma}\label{l.solid}
  The number of $\ell \in \{ 0,\ldots, n \}$ such that
  $$ 
  \PSAW_n \big( \Ga \closes \, \big\vert \, \lexm = \ell \big) \leq n^{ - \frac{1}{4} + 4 \eps}
  $$
  is at most $2n^{1-\ep}$.
\end{lemma}

\begin{proof}
  By the polygonal invariance (Lemma~\ref{lem.polyinv}), 
  $\lexm(\Ga)$ conditionally on $\Ga$ closing is uniform in $\{0, \dots, n\}$, so that, 
  for each $\ell \in \{0,\ldots,n\}$,
  \begin{equation}\label{eq:solid}
    \PSAW_n \big( \Ga \closes , \lexm = \ell \big) = \frac{1}{n+1} \PSAW_n \big( \Ga \closes \big).
  \end{equation}
  Hence, by \eqref{eq.assumption},
  \begin{align*}
    \PSAW_n \big( \Ga \closes \, \big\vert \,  \lexm = \ell \big) 
   &= \frac{ \PSAW_n (\Ga \closes , \lexm = \ell) } {\PSAW_n(\lexm = \ell)} \\
    & \geq \frac{n^{ - \frac{1}{4} + 5 \eps}}{n+1} \frac{1}{\PSAW_n (\lexm = \ell)},
  \end{align*}
  which in turn gives
  \begin{align*}
    n+1=
    \sum_{\ell = 0}^n{(n+1)\PSAW_n(\lexm = \ell)} 
    \ge \sum_{\ell=0}^n\frac{n^{ - \frac{1}{4} + 5 \eps}}{\PSAW_n \big( \Ga \closes  \, \big\vert \, \lexm = \ell \big) }.
  \end{align*}
  Since $n \geq 2$, the lemma follows.
\end{proof}

\begin{definition}
A walk $\ga\in \SAW_\ell$ with $\lexm(\ga)=\ell$ is said to be \emph{good} if 
\begin{align*}
  \PSAW_{n + \ell - \lo} \big( \Ga \closes \, \big\vert \, \Ga^1 = \ga \big) \geq n^{ - \frac{1}{4} + 3 \eps}.
\end{align*}
Any translate of a good walk is also called good.
\end{definition}
Thus, $\ga$ is good if, 
when completed with $n-\lo$ steps in such a way that the resulting walk has \lex\ time $\ell$, 
the resulting walk has a reasonable chance of closing.
Note that the above definition is specific to the values of $n$ and $\lo$ fixed before. 

In the next lemma, we bound  from below (still under the assumption that inequality~\eqref{eq.assumption} holds) the probability of being good
for the first part of a walk with $\hang$ close to $\lo$. 
We start with the case $\hang = \lo$, 
then we resample patterns using Lemma~\ref{lem:balanced} to change $\hang$ by an additive constant smaller than $\sqrt n$. 

\begin{lemma}\label{l.good}
  For $n$ large enough and any $0\le k\le \sqrt n$,
  \begin{align}
    & \PSAW_n \big( \Ga^1 \text{ is good} \, \big\vert \, \Ga \closes , \lexm = \lo-2k \big) \geq n^{ - \frac{1}{4} + 2 \eps}\label{eq.good2}.
  \end{align}
\end{lemma}

\begin{proof}
  We start with the case $k=0$. First note that
  $$\ESAW_n \Big[ \PSAW_n \big( \Ga \closes  \big|\Ga^1 \big) \Big|  |\Ga^1| = \lo  \Big] =  \PSAW_n \big( \Ga \closes \, \big| \,  \lexm = \lo \big) 
  \geq n^{ - \frac{1}{4} + 4 \eps}.$$  
  Since we may assume that $n \geq 2^{1/\eps}$,
  \begin{eqnarray}
    & & \PSAW_n \big( \Ga^1 \text{ is good} \, \big| \, \lexm = \lo \big) \label{eq.first} \\
    & = & \PSAW_n \left( \PSAW_n \big( \Ga \closes \, \big| \, \Ga^1 \big) \geq n^{ - \frac{1}{4} + 3 \eps} \, \Big| \, |\Ga^1| = \lo  \right)
    \geq  n^{ - \frac{1}{4} + 3 \eps}. \nonumber
  \end{eqnarray}
  But
  \begin{eqnarray}
    &  & \frac
    {\PSAW_n \big( \Ga^1 \text{ is good} \, \big\vert \, \lexm = \lo ; \Ga \closes \big) }
    {\PSAW_n \big( \Ga^1 \text{ is good} \, \big\vert \, \lexm = \lo \big) } \label{eq.second} \\
    & = &
    \frac
    {\PSAW_n (\Ga \closes \cond \Ga^1 \text{ is good} ; \lexm = \lo) }
    {\PSAW_n (\Ga \closes \cond \lexm = \lo) }
    \geq 1.    \nonumber
  \end{eqnarray}
  The inequality is a direct consequence of the definition of a good walk. 
  From (\ref{eq.first}) and (\ref{eq.second}), we deduce that  
  \begin{align}
    \PSAW_n \left(\Ga^1 \text{ is good} \,\big| \,  \Ga \closes ; \lexm = \lo \right) \geq n^{ - \frac{1}{4} + 3 \eps}, \label{eq.good_k0}
  \end{align}
  which is an improved version of \eqref{eq.good2} for $k = 0$.

  \medbreak
  Now we extend the result to general values of $k$. 
  For this we will use Lemma~\ref{lem:balanced}.
  
  First observe that, for a shell $\si$ and a walk $\ga\in\si$, 
  the \lex\ point of $\ga$ is entirely determined by $\si$ 
  (beware of the fact that this is only true for the point, not the index).
  
  For a walk $\ga$, let $S_1$ denote the slots of $\varsigma(\ga)$ between the origin and $\ga_\lexm$
  and $S_2$ those after $\ga_\lexm$. 
  (The type I and II patterns are such that $\ga_\lexm$ cannot be a vertex belonging to a pattern of either type.) 
  We say that $\ga$ is {\em balanced} if 
  $\left| T^1_I(\ga)-\frac{T_I |S_1|}{|S_1|+|S_2|}\right| \le \sqrt n (\log n)^{1/2 + \eps}$.
  
  Fix $\de,c > 0$ for which \eqref{eq.patterns2} holds. 
  Let $\mathcal G$ be the set of shells satisfying the assumptions of Lemma~\ref{lem:balanced} 
  and such that, if $\ga\in \SAW_n$ satisfies $\varsigma(\ga) \in \mathcal G$, 
  then $\ga^1$ is good and $\ga$ closes. 
  Call $\mathcal G _{\text{bal}}$ the set of shells $\si \in \mathcal G$ such that
  any $\ga \in \si$ with $\lexm(\ga) = \lo$ is balanced.
  
  Note that $S_1$ and $S_2$ depend on $\ga$ only via $\varsigma(\ga)$. 
  Also, whether $\ga^1$ is good and whether $\ga$ is closing may each be determined from $\varsigma(\ga)$ alone. 
  Thus, $\varsigma(\ga) \in \mathcal G$ as soon as~$\varsigma(\ga)$ satisfies the assumptions of Lemma~\ref{lem:balanced},
  $\ga^1$ is good and~$\ga$ closes. 
  Moreover, any two walks from~$\SAW_n$ with the same shell and hanging time are either both balanced or both not balanced. 
  Hence, for~$\ga\in\SAW_n$ with~$\lexm=\ell_0$ and $\shell(\ga) \in \mathcal G$, 
  the shell~$\varsigma(\ga)$ is in~$\mathcal G _{\text{bal}}$ as soon as~$\ga$ is balanced.
  
  It will be useful to note that, by \eqref{eq.assumption}, \eqref{eq.good_k0} and polygonal invariance,
  \begin{align*}
    \PSAW_n \left(\Ga^1 \text{ is good}, \Ga \closes \text{ and } \lexm = \lo \right) \geq n^{ - \frac32 + 7 \eps}.
  \end{align*}
  Thus, by the choice of $\de$,
  \begin{align*}
	\PSAW_n \left( \varsigma(\Ga) \notin \mathcal G  \,\big|\, \Ga^1 \text{ is good},\Ga \text{ closes} \text{ and }  \lexm = \lo \right) 
     \le 4e^{-cn}n^{\frac32}. 
  \end{align*}
  Using \eqref{eq.good_k0} again, this implies that for $n$ large enough, 
   \begin{align*}
	  \PSAW_n \left( \varsigma(\Ga) \in \mathcal G  \,\big|\, \Ga \text{ closes} \text{ and }  \lexm = \lo \right) 
     \geq \tfrac12 n^{ - \frac{1}{4} + 3 \eps}.
  \end{align*}
   By the first part of Lemma \ref{lem:balanced}, for $n$ large enough, 
  \begin{align*}
    \PSAW_n \left( \Ga \text{ not balanced} \,\big|\, \varsigma(\Ga) = \sigma \right) \leq n^{-\frac74}, 
    \quad \forall \sigma\in \mathcal G.
  \end{align*}
  Thus 
  \begin{align*}
     & \PSAW_n \left(\varsigma(\Ga) \in \mathcal G \setminus \mathcal G _{\text{bal}} \,\big|\, \Ga \closes \text{ and } \lexm = \lo \right)
     \\& \qquad \leq \frac{\PSAW_n \left( \Ga \text{ not balanced}  \,\big|\,  \varsigma(\Ga) \in \mathcal G  \right)}
     	{\PSAW_n \left(\varsigma(\Ga) \in \mathcal G \text{ and } \lexm = \lo \right)}
     \leq 2 n^{-\frac14 - 7 \eps}.
  \end{align*}
  The above inequalities yield 
  \begin{align*}
     \PSAW_n \left(\varsigma(\Ga) \in \mathcal G _{\text{bal}} \,\big|\, \Ga \closes  ,  \lexm = \lo \right)
     \geq \tfrac14 n^{ - \frac{1}{4} + 3 \eps},
  \end{align*}
  for $n$ large enough. 

  Let $\sigma \in \mathcal G_{\text{bal}}$ and $\la$ be the number of type I patterns 
  needed in the first part of a walk $\ga \in \sigma$ so that $\lexm(\ga) = \lo$. 
  By \eqref{eq.resample}, for $0 \leq k \leq \sqrt n$ and $n$ large enough,  
  \begin{align*}
    \frac{\PSAW_n \big( \lexm= \lo - 2k \big| \varsigma= \sigma \big)}
    {\PSAW_n \big( \lexm = \lo  \big| \varsigma = \sigma \big)}
    = \frac{\PSAW_n \big( T_I^1(\Ga) = \la + k \big| \varsigma = \sigma \big)}
    {\PSAW_n \big( T_I^1(\Ga) = \la \big| \varsigma= \sigma \big)}
    \geq 4 n^{-\eps}.
  \end{align*}
  But
  \begin{eqnarray*}
    &&  \PSAW_n \big( \Ga^1 \text{ is good} \, \big\vert \, \Ga \closes , \lexm = \lo-2k \big) \\
    &&   \geq  \sum_{\si \in \mathcal G _{\text{bal}}} 
    \PSAW_n \big( \varsigma(\Ga) = \sigma \,  \big\vert \, \Ga \closes , \lexm = \lo-2k \big) \\
    &&   =   \sum_{\si \in \mathcal G _{\text{bal}}} 
    \PSAW_n \big( \lexm = \lo-2k \, \big\vert \,  \varsigma(\Ga) = \sigma \big) 
    \frac{\PSAW_n ( \varsigma(\Ga) = \sigma)}{\PSAW_n (\Ga \closes , \lexm = \lo-2k) }\\
    &&   \geq  4 n^{-\eps} \sum_{\si \in \mathcal G _{\text{bal}}} 
    \PSAW_n \big( \lexm = \lo  \, \big\vert \,  \varsigma(\Ga) = \sigma \big) 
    \frac{\PSAW_n(\varsigma(\Ga) = \sigma)}{\PSAW_n (\Ga \closes , \lexm = \lo) }\\
    &&  =  4 n^{-\eps} \PSAW_n \big( \varsigma(\Ga) \in  \mathcal G _{\text{bal}} \, \big\vert \, \Ga \closes , \lexm = \lo \big)   \\
    &&  \geq   n^{ - \frac{1}{4} + 2 \eps}.
  \end{eqnarray*}
  Here, we used polygonal invariance (Lemma~\ref{lem.polyinv}) to assert that 
  $$\PSAW_n (\Ga \closes , \lexm = \lo - 2k) = \PSAW_n (\Ga \closes , \lexm = \lo).$$
\end{proof}

\begin{definition}
For a closing walk $\ga \in \SAW_n$, an index $\ell$ is said to be \emph{ticked} if 
$\ga[\lexm - \ell, \lexm]$ is good.
\end{definition}
In this definition, closing walks are viewed as polygons and we use the modulo $n+1$ notation for their indices. 
Thus $\ga[\lexm - \ell, \lexm]$ may contain the edge $(\ga_n, \ga_0)$.
Note that the hanging point and the ticked indices of a closing walk only depend on the corresponding polygon.

Let $\sfT = \sfT(\ga)$ be the number of ticked indices in $\{\lo - 2k,0\le k\le \sqrt n\}$.
The next lemma shows that the probability of having many ticked points is not too small.
\begin{lemma}\label{l.ticked}
  For $n$ large enough,
  \begin{align*}
    \PSAW_n \Big( \sfT(\Ga) \geq  n^{\frac{1}{4} + \eps} \, \Big\vert \,  \Ga \closes \Big) \geq n^{-  \frac{1}{4} + \eps}.
  \end{align*}
\end{lemma}

\begin{proof}
 Consider $n$ large enough so that Lemma \ref{l.good} holds. For $0 \le k \le \sqrt n$,
  \begin{eqnarray*}
    &  & \PSAW_n \big( \lo - 2k \text{ is ticked} \, \big\vert \, \Ga \closes \big) \\
    & = & \PSAW_n \big(  \Ga^1 \text{ is good} \, \big\vert \,  \Ga \closes , \lexm = \lo - 2k \big) 
    \geq n^{ - \frac{1}{4} + 2 \eps},
  \end{eqnarray*}
  where the equality is due to polygonal invariance (Lemma~\ref{lem.polyinv}) and the inequality to Lemma \ref{l.good}. 
  It follows that 
  \begin{align*}
    \ESAW_n \left[ \sfT(\Ga) \, \big\vert \, \Ga \closes \right]
    = \sum_{k=0}^{\sqrt n}\PSAW_n \big( \lo - 2k \text{ is ticked} \, \big\vert \, \Ga \closes \big)
    \geq n^{ \frac{1}{4} + 2 \eps}.
  \end{align*}
 Since $\sfT$ is bounded by $1+\sqrt n$ and $n \geq 4^{1/\eps}$, we find that
  \begin{align*}
    \PSAW_n \Big( \sfT(\Ga) \geq n^{-\frac{1}{4} + \eps} \sqrt n \, \Big\vert \, \Ga \closes \Big)
    \geq n^{ - \frac{1}{4} + \eps}. 
  \end{align*}
\end{proof}

The next lemma shows that a portion of walk with many ticked indices, ending at some site~$z$, 
is very unlikely to be the beginning of a self-avoiding walk whose hanging point is~$z$.
\begin{lemma}\label{l.avoidance}
  For a closing walk $\chi \in \SAW_n$ with $\sfT(\chi) \geq n^{\frac{1}{4} + \eps}$,
  \begin{align}
    \PSAW_{n} \left( \lexm(\Ga) =  \lo \, \Big\vert \, \Ga[0, \lo] = \chi[\lexm - \lo, \lexm] \right) 
    &\leq e^{ - n^{\eps}}.
    \label{eq.avoided}
  \end{align}
\end{lemma}

\begin{proof}
  Let $\chi \in \SAW_n$ be a closing walk with $\sfT(\chi) \geq n^{\frac{1}{4} + \eps}$ and  
  assume without loss of generality that $\lexm(\chi) = n$.

  For the purpose of this proof only,
  let $ \Wtmp$ be the set of walks $\ga$ of length~$n-\lo$,
  originating at $\chi_n$, with $\lexm(\ga) = 0$.
  Let $\Ptmp$ denote the uniform measure on the set $\Wtmp$. When working with  $\Ptmp$, 
  $\Gamma$ denotes a random variable distributed according to $\Ptmp$. 
  In particular, $\Gamma$ is contained in the half-space below and including $\chi_n$.
  
  We now extend the notion of closing walk by saying that  $\ga'$ {\em closes} $\ga$ if $\ga_{|\ga|} = \ga'_0$ 
  and $\ga'_{|\ga'|}$ is adjacent to $\ga_0$. 
  We say that $\ga'$ {\em avoids} $\ga$ if $\ga'\cap \ga = \{\ga'_0\}$.
  
  Let $ t_1 < \dots < t_{\sfT}$  be the ticked indices of $\chi$ contained in $\{\lo - 2k, 0\le k \le \sqrt n\}$.
  Consider the walks $\chi^{j} = \chi[n - j , n]$ for $0 \leq j \leq n$.
  They all end at $\chi_n$ and $ \chi^{j} \subsetneq \chi^{j+1}$. 
  For $ 1 \leq i \leq \sfT$, define
  \begin{align*}
    A_i&=\{ \Gamma \text{ avoids } \chi^{t_i} \} \quad \text{ and } \quad C_i=\{\Gamma\text{ closes }\chi^{t_i}\}.
  \end{align*}
  Also, let $\Atmp = \{ \ga \in \Wtmp : \ga \text{ avoids } \chi^{\ell_0}\}$.
   
  Since $\chi^{t_i}$ is good, 
  \begin{align}\label{eq.closing_avoiding}
    \Ptmp \left(\Gamma \closes \, \chi^{t_i} \mcond \Gamma \text{ avoids } \chi^{t_i}\right)
    = \Ptmp \left(C_i \mcond  A_i \right) 
    \geq n^{ - \frac{1}{4} + 3 \eps}.
  \end{align}
  Write $k = \lceil  4d n^{\frac{1}{4} - 3 \eps} \rceil$ and suppose that $k \leq \sfT$. 
  Any realization $\Ga \in \Wtmp$ is in at most $2d$ events $C_i$.
  Hence, by \eqref{eq.closing_avoiding} and the fact that the~$A_j$ are decreasing, 
  \begin{align*}
    2d 
    \geq \sum_{i=1}^{k} \Ptmp (C_i)
    \geq \sum_{i=1}^{k} \Ptmp (C_i \cond A_i)\Ptmp(A_k)
    \geq 4d \Ptmp(A_k).
  \end{align*}
  Therefore, $\Ptmp (A_k) \leq \frac{1}{2}$.
  If the procedure is repeated between $k + 1$ and $2k$, one obtains
  \begin{align*}
    2d
    \geq \sum_{i= k+1}^{2k} \Ptmp (C_i \cond A_k)
    \geq \sum_{i=k+1}^{2k} \Ptmp (C_i \cond A_i) \Ptmp(A_{2k} \cond A_k)
    \geq 4d \Ptmp(A_{2k} \cond A_k),
  \end{align*}
  and thus $P(A_{2k} \cond A_k ) \leq 1/2$.
  Since $A_{2k}\subset A_k$, we find 
  $$\Ptmp (A_{2k}) = \Ptmp (A_k)\Ptmp (A_{2k} \cond A_k) \leq \tfrac14.$$
  This procedure may be repeated $ \lfloor \frac{\sfT}{k}\rfloor $ times.
  Since $\sfT \geq n^{\frac{1}{4} + \eps}$, we obtain 
  \begin{align*}
    \frac{|\Atmp|}{|\Wtmp|}=\Ptmp(\Gamma\text{ avoids }\chi^{\ell_0})\le \Ptmp (A_\sfT) 
    \leq 2^{- \lfloor \frac{\sfT}{k}\rfloor} 
    \leq 2^{ - \tfrac{n^{4 \eps}}{2(4d+1)}},
  \end{align*}
  for $n$ large enough and $\eps \leq 1/12$, which can be harmlessly assumed. 
  
  \begin{figure}
    \begin{center}
      \cpsfrag{l}{$n - \lo$}
      \cpsfrag{ti}{$n - t_1$}
      \cpsfrag{tii}{$n - t_2$}
      \cpsfrag{tt}{$n - t_\sfT$}
      \cpsfrag{cn}{$\chi_n$}
      \cpsfrag{e}{$e_1$}
      \includegraphics[width=0.75\textwidth]{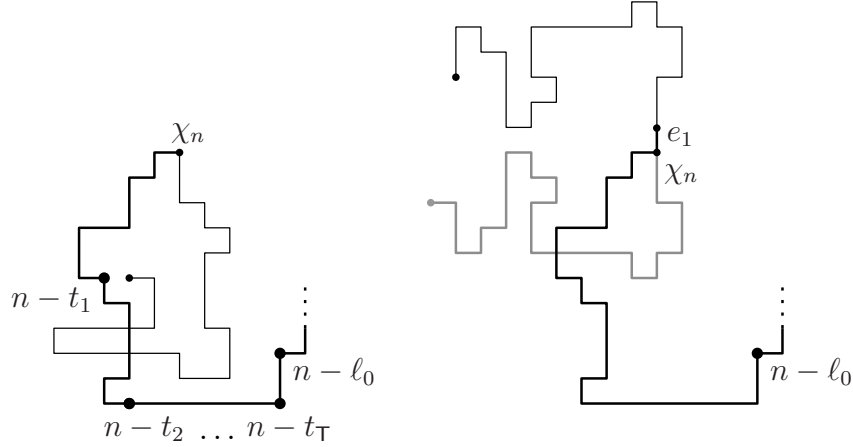}
    \end{center}
    \caption{\emph{Left:} The final portion of $\chi$ in bold, 
      and a walk contained in $C_1$ and in $A_1$ but not in $A_2$. 
      \emph{Right:} A walk $\ga \in \Wtmp \setminus \Atmp$ (gray) may be reflected 
      and added to $\chi^{\lo}$ (black)
       to create a walk starting with $\chi^{\lo}$, with 
      \lex\ point different from $\chi_n$.
      When concatenating $\Rfl_{\chi_n} (\ga)$ to $\chi^{\lo}$, an extra vertical edge is added in 
      to ensure non-intersection, and the last edge is deleted to preserve length $n$. }
    \label{fig:avoided}
  \end{figure}

  Let us now express the probability in \eqref{eq.avoided} in terms of the ratio $|\Atmp|/|\Wtmp|$. 
  The set $\Atmp$ contains all the possible continuations $\ga$ of $\chi^{\lo}$ 
  for which $\chi^{\lo} \circ \ga$ is a self-avoiding walk of length $n$ with $\lexm = \lo$.
  On the other hand, for $\ga \in \Wtmp$, 
  the walk obtained by concatenating to $\chi^{\lo}$ 
  an edge $e_1$ followed by 
  $\Rfl_{\chi_n} (\ga)$ is a self-avoiding walk of length $n + 1$ with $\lexm > \lo$. 
  By deleting the last edge of such walks, we obtain at least $|\Wtmp| / 2d$
  walks of length $n$, starting with $\chi^{\lo}$ and having $\lexm \neq \lo$.
  See Figure \ref{fig:avoided}.
  Thus, 
  \begin{align*}
    \PSAW_n \big( \lexm(\Ga) =  \lo \, \big\vert \, \Ga[0, \lo] = \chi^{\lo} \big)
    &\leq \frac{ 2d |\Atmp|}{|\Wtmp|}\leq 2d  2^{ - \tfrac{n^{4 \eps}}{2(4d+1)}}
    \leq e^{ - n^{\eps}}
  \end{align*}
  for $n$ large enough. This proves \eqref{eq.avoided}.
\end{proof}
  
We are now ready to conclude the proof of Theorem~\ref{t.nonclosing}. 
A walk~$\chi \in \SAW_{\lo}$ with~$\lexm (\chi)=\lo$ is called \emph{untouchable} if
$$
\PSAW_{n} \left( \lexm(\Ga) =  \lo \, \Big\vert \, \Ga[0, \lo] = \chi \right) \leq e^{ - n^{\eps}}.
$$
By Lemmas \ref{l.ticked}, \ref{l.avoidance} and polygonal invariance,
\begin{align*}
  \PSAW_{n} \big( \Ga^1 \text{ is untouchable} \, \big\vert \, \Ga \closes , \lexm = \lo \big) 
  \geq n^{- \frac{1}{4} + \eps}.
\end{align*}
Hence 
\begin{eqnarray*}
  \frac{n^{-\frac14+5\ep}}{n+1} 
  &\leq &
  \PSAW_{n} \big( \Ga \closes , \lexm = \lo \big) \\
  & \leq &
  \frac
  {\PSAW_{n} \big(\Ga \closes , \lexm = \lo \big)}
  {\PSAW_{n} \big( \Ga[0, \lo] \text{ is untouchable}\big)}  \\
 & & \qquad \times \, 
  \frac
  {\PSAW_{n} \big(\Ga[0, \lo] \text{ is untouchable},  \lexm = \lo  \big) }
  {\PSAW_{n} \big( \Ga[0, \lo] \text{ is untouchable},\, \Ga \closes , \lexm = \lo  \big)}\\
  & = &  
  \frac
  {\PSAW_{n} \big( \lexm = \lo \cond \Ga[0, \lo] \text{ is untouchable} \big)}
  { \PSAW_{n} \big( \Ga[0, \lo] \text{ is untouchable} \cond \Ga \closes , \lexm = \lo  \big)} \\
  &\leq &   
  e^{- n^{\eps}} n^{\frac{1}{4} - \eps}.
\end{eqnarray*}
This is a contradiction for $n$ large enough, and the proof of Theorem \ref{t.nonclosing} is complete. \qed

\bibliographystyle{plain}

\bibliography{saw}

\end{document}